\documentclass[12pt,reqno]{amsart}
\usepackage{amsmath,amsfonts,amssymb,mathrsfs,amstext,amscd,latexsym,amsthm}
\usepackage{comment}
\usepackage{mathtools}
\usepackage{hyperref}
\usepackage{enumitem}
\newcommand{\R}{\ensuremath{\mathbb{R}}}

\newcommand{\N}{\ensuremath{\mathbb{N}}}

\newcommand{\mcf}{\eqref{eq:MCF}\ }

\newcommand{\Div}{\mathrm{div}} 

\newcommand{\dist}{\mathrm{dist}} 
\newcommand{\diag}{\mathrm{diag}} 

\newcommand{\tr}{\mathrm{tr}}
\newcommand{\diam}{\mathrm{diam}}

\newcommand{\pd}{\partial}
\newcommand{\cd}{\nabla}
\newcommand{\inner}[2]{\left\langle #1 \, , \, #2\right\rangle} %Euclidean inner product
\newcommand{\norm}[1]{\left\Vert#1\right\Vert} %Euclidean norm
\newcommand{\M}{M} %{{}\hspace{-0.5pt}{}\mathscr{M}{}\hspace{-0.5pt}{}} %flow parametrisation 
\newcommand{\X}{X} %{{}\hspace{-1.5pt}{}\mathscr{X}{}\hspace{-1.5pt}{}} %flow parametrisation 
\newcommand{\W}{A} %{\mathcal{W}} %Weingarten map
\newcommand{\lb}{\left(}
\newcommand{\rb}{\right)}
\newcommand{\lsb}{\left[}
\newcommand{\rsb}{\right]}
\newcommand{\lcb}{\left\{}
\newcommand{\rcb}{\right\}}

\newcommand{\Ges}{G_{\varepsilon,\sigma}}
\newcommand{\Gesp}{G_{\varepsilon,\sigma,+}}
\newcommand{\Gesk}{G_{\varepsilon,\sigma,K}}
\newcommand{\Geskp}{G_{\varepsilon,\sigma,K,+}}

\renewcommand{\AA}{\textit{\r{A}}}
\def\labelitemi{--}
\def\ba #1\ea {\begin{align} #1\end{align}}
\def\bann #1\eann {\begin{align*} #1\end{align*}}
\def\ben #1\een {\begin{enumerate} #1\end{enumerate}}
\def\bi #1\ei {\begin{itemize}\renewcommand\labelitemi{--} #1\end{itemize}}
\theoremstyle{plain}
\numberwithin{equation}{section}
\newtheorem{thm}{Theorem}[section]
\newtheorem*{thm*}{Theorem}

\newtheorem{lem}[thm]{Lemma}

\newtheorem{cor}[thm]{Corollary}
\newtheorem{claim}[thm]{Claim}
\newtheorem{prop}[thm]{Proposition}

\newtheorem*{conds*}{Conditions}
\newtheorem*{auxconds*}{Ancillary Conditions}
\newtheorem*{props*}{Properties}
\theoremstyle{remark}
\newtheorem{rem}{Remark}[section]

\usepackage{cite}

\title[Sharp one-sided curvature estimates for MCF]{Sharp one-sided curvature estimates for mean curvature flow.}

\author{Mat Langford}
\date{\today}

\begin{document}

\begin{abstract}
We prove a sharp pinching estimate for immersed mean convex solutions of mean curvature flow which unifies and improves all previously known pinching estimates, including the umbilic estimate of Huisken \cite{Hu84}, the convexity estimates of Huisken--Sinestrari \cite{HuSi99b} and the cylindrical estimates of Huisken--Sinestrari \cite{HuSi09} (see also \cite{AnLa14,HuSi15}). Namely, we show that the curvature of the solution pinches onto the convex cone generated by the curvatures of any shrinking cylinder solutions admitted by the initial data. For example, if the initial data is $(m+1)$-convex, then the curvature of the solution pinches onto the convex hull of the curvatures of the shrinking cylinders $\mathbb{R}^m\times S^{n-m}_{\sqrt{2(n-m)(1-t)}}$, $t<1$. In particular, this yields a sharp estimate for the largest principal curvature, which we use to obtain a new proof of a sharp estimate for the inscribed curvature for embedded solutions \cite{Br15,HaKl15,La15}. Making use of a recent idea of Huisken--Sinestrari \cite{HuSi15}, we then obtain a series of sharp estimates for ancient solutions. In particular, we obtain a convexity estimate for ancient solutions which allows us to strengthen recent characterizations of the shrinking sphere due to Huisken--Sinestrari \cite{HuSi15} and Haslhofer--Hershkovitz \cite{HaHe}. %In a related work (with Lynch) we use similar techniques to prove analogous results for various curvature flows by fully non-linear speeds \cite{LaLy}.
\end{abstract}

\maketitle
\tableofcontents

\section{Introduction}

Let $M^n$ be a smooth manifold of dimension $n$ and $I\subset \R$ an interval. A smooth one-parameter family $X:M^n\times I\to \R^{n+1}$ of smooth immersions $X(\cdot,t):M^n \to \R^{n+1}$ \emph{evolves by mean curvature flow} if its velocity at each point is given by its mean curvature at that point; that is, if
\begin{equation}\label{eq:MCF}\tag{MCF}
\pd_tX(x,t) = \vec H(x,t)
\end{equation}
for every $(x,t)\in M^n\times I$, where $\vec H=\Div(DX)$ is the mean curvature vector of the immersion. We will be interested in mean curvature flow of \emph{mean convex} (resp. \emph{strictly mean convex}) hypersurfaces; that is, two-sided hypersurfaces such that, with respect to one of the two choices of unit normal field $\nu$, the mean curvature $H=-\vec H\cdot \nu$ is non-negative (resp. positive). We will often say that a solution of \eqref{eq:MCF} is a compact/mean convex/etc. solution if all the time slices $M^n_t:=X_t(M^n)$, where $X_t:=X(\cdot,t)$, are compact/mean convex/etc. Unless otherwise stated, we allow the possibility that the solution consists of multiple connected components.

%In the first part of the paper, we consider solutions $X:M^n\times I\to\R^{n+1}$ of \mcf arising from some initial immersion $X_0:=X(\cdot,0):M^n\to\R^{n+1}$, so that $I$ is of the form $[0,T)$ for some $T<\infty$. 

A fundamental tool in the analysis of solutions of \mcf is the parabolic maximum principle. A basic consequence is the preservation of mean convexity under the flow, since the mean curvature $H$ satisfies the \emph{Jacobi equation}
\ba\label{eq:evolveH}
(\pd_t-\Delta)H=|A|^2H\,,
\ea
where $\Delta$ is the Laplace--Beltrami operator, $A$ is the second fundamental tensor and $| \cdot |$ the norm corresponding to the induced geometry. We will also be interested in other convexity conditions such as (\emph{strict}) \emph{convexity} and (\emph{strict}) $k$-\emph{convexity}, where we recall that a hypersurface is called \emph{convex}\footnote{If $M^n=\pd K$ for some convex body $K\subset\R^{n+1}$, we will say, explicitly, that $M^n$ \emph{bounds a convex body}.% Note that we have made no connectedness assumption on $M^n$.
} (resp. \emph{strictly convex}) if its Weingarten tensor $A$ is everywhere non-negative definite (resp. positive definite) %and \emph{mean convex} (resp. \emph{strictly mean convex}) if it is two-sided and, with respect to one of the two choices of unit normal, $\nu$, the mean curvature $H:=-\vec H\cdot\nu$ is non-negative (resp. positive).
%A further interesting class is given by the $k$-\emph{convex} hypersurfaces; here, we call a hypersurface
and \emph{$k$-convex} (resp. \emph{strictly $k$-convex}) for some $k\in\{1,\dots n\}$ if the sum $\kappa_1+\dots+\kappa_k$ of its smallest $k$ principal curvatures is everywhere non-negative (resp. positive). %and \emph{uniformly $k$-convex} if it is mean convex and there is a constant $\alpha>0$ such that $\kappa_1+\dots+\kappa_k\geq \alpha H$. 
%The $k$-convexity condition interpolates between mean convexity ($k=n$) and convexity %\footnote{What we refer to as \emph{convexity} is usually referred to as \emph{local convexity}, the term \emph{convex} being reserved for boundaries of convex bodies. Since $M^n$ is compact, the two notions are equivalent if $X$ is an embedding or if $n\geq 2$.} 
%($k=1$). 
More generally, we will consider $n$-dimensional hypersurfaces whose Weingarten tensors, after choosing an orthonormal basis, lie, at every point, inside some convex subset $\Gamma\subset \mathcal{S}(n)$ in the space $\mathcal{S}(n)$ of self-adjoint endomorphisms of $\R^n$. In order to ensure that the condition is independent of the choice of basis, we should require that $\Gamma$ is \emph{$O(n)$-invariant}; that is, invariant under the action of $O(n)$ by conjugation. Moreover, since rescaling limits at singularities tend to move purely by scaling, we should also require that $\Gamma$ is a \emph{cone}; that is, invariant under scaling by positive numbers. A more subtle application of the maximum principle---to the evolution equation for the second fundamental form
\ba\label{eq:evolveA}
(\cd_t-\Delta)A=|A|^2A\,,
\ea
where $\cd_t$ is the covariant time derivative (see e.g. \cite[\S 6.3]{AH11})---reveals that such sets are preserved by \eqref{eq:MCF} (see e.g. \cite[\S 4 and \S 8]{Hm86}, \cite[\S 7.3]{AH11} or \S \ref{sec:prelims} below). %defined on tangent vector fields by
%\bann
%\cd_tv:=[\pd_t,v]-HA(v)\,,
%\eann
 In particular, $k$-convexity is preserved for any $k\in\{1,\dots,n\}$.%, where we recall that a hypersurface is called \emph{$k$-convex}, for some $k\in\{1,\dots n\}$, if the sum $\kappa_1+\dots+\kappa_k$ of its smallest $k$ principal curvatures is everywhere non-negative, and \emph{uniformly $k$-convex} if it is mean convex and there is a constant $\alpha>0$ such that $\kappa_1+\dots+\kappa_k\geq \alpha H$. The $k$-convexity condition interpolates between mean convexity ($k=n$) and convexity\footnote{What we refer to as \emph{convexity} is usually referred to as \emph{local convexity}, the term \emph{convex} being reserved for boundaries of convex bodies. Since $M^n$ is compact, the two notions are equivalent if $X$ is an embedding or if $n\geq 2$.} ($k=1$). 

\subsection{One sided curvature pinching}

%A fundamental problem in the study of mean curvature flow is to quantitatively describe the behaviour of solutions near a singularity. The first step to obtain `one-sided pinching estimates' for the curvature. 

Mean curvature flow of convex hypersurfaces in dimensions $n\geq 2$ was studied in a groundbreaking paper of Huisken \cite{Hu84}, where it was proved that such hypersurfaces shrink to `round' points. A crucial part of the analysis was the \emph{umbilic estimate}, which states that given any $\varepsilon>0$ there is a constant $C_\varepsilon$, which depends only on $\varepsilon$ and the initial data, such that
\ba\label{eq:umbilic}
|\AA|^2\leq \varepsilon H^2+C_\varepsilon\,,
\ea
where $\AA$ denotes the trace free part of $A$. In particular, this implies that the scaling invariant ratio $|\AA|^2/H^2$ is becoming arbitrarily small wherever $H$ is blowing up. In other words, the hypersurface is, modulo rescaling, umbilic at a singularity. To prove such an estimate, one attempts to bound the function
\bann
f_\sigma:=\frac{|\AA|^2}{H^2}H^\sigma
\eann
for some small $\sigma>0$. Huisken achieves this by bounding the $L^p$-norms of $f_\sigma$ for large $p$ and $\sigma\approx p^{-\frac{1}{2}}$ and applying a Stampacchia-type iteration with the help of the Michael--Simon Sobolev inequality. This method turns out to be quite robust and variants of the argument were later applied to obtain curvature estimates in the non-convex setting. The next breakthrough was the \emph{convexity estimate} \cite{HuSi99a,HuSi99b} (see also \cite{Wh03}), which states that, if the initial immersion is strictly mean convex, then given any $\varepsilon>0$ there is a constant $C_\varepsilon$, which depends only on $\varepsilon$ and the initial data, such that
\ba\label{eq:convexity}
\kappa_1\geq -\varepsilon H-C_\varepsilon\,.
\ea
This estimate implies that the scaling invariant tensor $A/H$ is non-negative definite at a singularity, which markedly constrains the geometry of the hypersurface at such points. Interpolating between the umbilic and convexity estimates \eqref{eq:umbilic} and \eqref{eq:convexity} are the $m$-\emph{cylindrical estimates} \cite{HuSi09,AnLa14,HuSi15}: If the initial immersion is $(m+1)$-convex, $m\in\{0,\dots,n-1\}$, then given any $\varepsilon>0$ there is a constant $C_\varepsilon$, which depends only on $\varepsilon$ and the initial data, such that
\ba\label{eq:mcylindrical}
|A|^2-\frac{1}{n-m}H^2\leq \varepsilon H^2+C_\varepsilon\,.
\ea
%If we denote by $\overline A$ the restriction of the second fundamental form to the eigenspace of the eigenvalues $\{\kappa_{m+1},\dots,\kappa_n\}$ and $\overline H$ its trace, then
%\bann
%|A|^2-\frac{1}{n-m}H^2={}&|\overline A|^2-\frac{1}{n-m}\overline H^2+(\kappa_1^2+\dots+\kappa_m^2)\\
%{}&-\frac{1}{n-m}(\kappa_1+\dots+\kappa_m)(H+\overline H)\,.%\\
%\geq{}&(\kappa_1^2+\dots+\kappa_m^2)-\frac{1}{n-m}(\kappa_1+\dots+\kappa_m)(H+\overline H)\,.
%\eann
On a convex hypersurface, the left hand side of \eqref{eq:mcylindrical} is non-positive only at points which are either strictly $m$-convex, $\kappa_1+\dots+\kappa_m>0$, or \emph{$m$-cylindrical}, $0=\kappa_1=\dots=\kappa_m$ and $\kappa_{m+1}=\dots=\kappa_n$. This is particularly restrictive in case $m=1$ and $n\geq 3$. Indeed, in that case, up to rescaling, every singularity has the geometry of either shrinking sphere $S^n_{\sqrt{2n(1-t)}}$, a shrinking cylinders $\R\times S^{n-1}_{\sqrt{2(n-1)(1-t)}}$ or the strictly convex, rotationally symmetric translating bowl \cite{HuSi09,Ha,LB}. Note that the case $m=0$ is just the umbilic estimate \eqref{eq:umbilic} and the case $m=n-1$ follows from the convexity estimate \eqref{eq:convexity}. 

%A somewhat different version of these estimates were proved in \cite{AnLa14} (for flows by a class of non-linear functions of curvature which includes the mean curvature flow). There, one obtains for every $\varepsilon>0$ a constant $C_\varepsilon$, which depends only on $\varepsilon$ and the initial data, such that
%\ba\label{eq:mcylindrical2}
%\kappa_1+\dots+\kappa_{m+1}-\frac{1}{n-m}H\geq-\varepsilon H-C_\varepsilon\,.
%\ea
%This estimate has analogous consequences because
%\bann
%(n-m)\sum_{i=1}^{m+1}\kappa_i-H=(n-m-1)\sum_{i=1}^m\kappa_i-\sum_{j=m+1}^n(\kappa_j-\kappa_{m+1})\,.
%\eann
We also note that, somewhat surprisingly, the iteration method can even be applied in the setting of fully non-linear curvature flows (by isotropic, 1-homogeneous flow speeds), so long as the speed satisfies appropriate structure conditions. Indeed, flows by convex speed functions tend to admit lower curvature pinching \cite{ALM14,AnLa14}, whereas flows by concave speed functions tend to admit upper curvature pinching \cite{L,BrHu}%,LaLy}
. Moreover, no concavity assumption is necessary in two space dimensions \cite{An10,ALM15}.

%Observe that each of the pinching estimates described above estimates the distance of the second fundamental tensor to a convex, $O(n)$-invariant cone in the normed linear space $\mathcal{S}(n)$ of self-adjoint endomorphisms of $\R^n$. Indeed, these cones are
%\bann
%\Gamma_{0}={}&\{\lambda \mathrm{I}\in \mathcal{S}(n):\lambda\geq 0\}\\
%\Gamma_{+}={}&\{W\in \mathcal{S}(n):W\geq 0\}\quad \mbox{and}\\ 
%\Gamma_{m}={}&\lcb W\in \mathcal{S}(n): |W|^2\leq \frac{1}{n-m}\tr(W)^2\rcb\,.
%\eann
%The respective estimates can then be written in the form
%\bann
%\dist(A_{(x,t)},\Gamma)\leq \varepsilon H(x,t)+C_\varepsilon\,,
%\eann
%where $\Gamma$ is one of the above cones and, since each of the cones is $O(n)$-invariant, the distance is can be defined with respect to any choice of basis.

In \S \ref{sec:pinching}, we prove a unified and sharp result of the above form. Namely, we show that the curvature of a mean convex solution pinches onto the convex cone generated by the curvatures of any shrinking cylinders not ruled out by the initial curvature hull. To state the result precisely, we refer to an open, convex, $O(n)$-invariant cone $\Gamma\subset \mathcal{S}(n)$ in the normed linear space $\mathcal{S}(n)$ of self-adjoint endomorphisms of $\R^n$ as a \emph{pinching condition} and call a point $W\in \mathcal{S}(n)$ \emph{cylindrical} if, for some $m\in\{0,\dots,n\}$, $W$ has a null eigenvalue of multiplicity $m$ and a positive eigenvalue of multiplicity $(n-m)$. We will also abuse notation by writing $A_{(x,t)}\in \Gamma$ if this is true after identifying $(T_xM,g_{(x,t)})$ with $(\R^n,\inner{\cdot}{\cdot}_{\R^n})$ by choosing some (and hence any) orthonormal basis for $(T_xM,g_{(x,t)})$.

\begin{thm}[Pinching principle]\label{thm:pinching}
Fix a dimension $n\in \N\setminus\{1\}$ and a pinching condition $\Gamma\subset \mathcal{S}(n)$ and denote by $\Lambda$ the convex hull of the cylindrical points in $\Gamma$. Then, given any pinching condition $\Gamma_0\subset \mathcal{S}(n)$ satisfying $\overline \Gamma_0\setminus\{0\}\subset \Gamma$, a curvature scale $\Theta<\infty$, and any $\varepsilon>0$, there is a constant $C_\varepsilon=C_\varepsilon(n,\Gamma_0,\Theta,\varepsilon)<\infty$ with the following property: Let $X:M^n\times[t_0,T)\to\R^{n+1}$ be a compact solution of \mcf satisfying
\ben
\item $\displaystyle \lb\frac{\mu_{t_0}(M^n)}{\sigma_n}\rb^{\frac{1}{n}}+\sqrt{2n(T-t_0)}\leq 2R$, where $\sigma_n:=\mathrm{Area}(S^n)$,
\item $\displaystyle \max_{M^n\times\{t_0\}}H\leq \Theta R^{-1}$ and
\item $\displaystyle A_{(x,t_0)}\in \overline\Gamma_0$ for all $x\in M^n$.
\een
Then 
\ba\label{eq:pinching}
\dist(A_{(x,t)},\Lambda)\leq \varepsilon H(x,t)+C_\varepsilon R^{-1}
\ea
for all $(x,t)\in M^n\times[t_0,T)$.
\end{thm}

Of course, any compact solution of \mcf arising from an initial immersion $X_0:M^n\to\R^{n+1}$ satisfying $A_{(x,0)}\in \Gamma$ for all $x\in M^n$ satisfies each of the conditions (1)--(3) for some such $R$, $\Theta$ and $\Gamma_0$. By the strong maximum principle, this is also the case, after waiting a short time, under the initial condition $A_{(x,0)}\in \overline\Gamma$. %A typical example of a pinching condition $\Gamma$ is given by the $k$-convexity condition discussed above.

Theorem \ref{thm:pinching} is asymptotically sharp: Fix $\Gamma$ and choose $m$ so that the curvature of $\R^m\times S^{n-m}$ lies in $\pd\Lambda$. Then given data $n$, $\Theta$ and $\Gamma_0$ (containing the curvature of $\R^m\times S^{n-m}$) there is a sequence of compact solutions $X_i:M_i^n\times [t_i,1)\to\R^{n+1}$ of \mcf satisfying (1)--(3) with $R_i\to\infty$ which converge locally uniformly to the shrinking cylinder solution $\R^m\times S^{n-m}_{\sqrt{2(n-m)(1-t)}}$, $t\in (-\infty,1)$.

The proof of Theorem \ref{thm:pinching} boils down to two rather simple observations: First, the (signed) distance of the Weingarten curvature to the boundary of a convex cone in $\mathcal{S}(n)$ is a supersolution of the Jacobi equation (Proposition \ref{prop:evolveG}) %\bann
%(\pd_t-\Delta)u=|A|^2u
%\eann
 and, second, the crucial estimates needed to set up the Stammpachia iteration argument hold away from cylindrical points (Lemma \ref{lem:gradterm} and Proposition \ref{prop:Poincare}). %These observations are already more or less known. The former provides the basis of the proof of the tensor maximum principle mentioned above and the latter is a consequence of certain properties of two interesting tensors (closely related observations were made in \cite{Hu84} and \cite{BrHu15}).

Each of the aforementioned one-sided curvature estimates is an easy corollary of Theorem \ref{thm:pinching}. The following corollary is not implied by the previous estimates.

\begin{cor}[$m$-convexity estimate]\label{cor:mconvexity}
Given a dimension $n\in\N\setminus\{1\}$, a pinching constant $\alpha>0$, a curvature scale $\Theta<\infty$ and any $\varepsilon>0$, there is a constant $C_\varepsilon<\infty$ with the following property: Let $X:M^n\times[t_0,T)\to\R^{n+1}$ be a compact solution of \mcf satisfying, for some $m\in\{0,\dots,n-2\}$,
\ben
\item $\displaystyle \lb\frac{\mu_{t_0}(M^n)}{\sigma_n}\rb^{\frac{1}{n}}+\sqrt{2n(T-t_0)}\leq 2R$, where $\sigma_n:=\mathrm{Area}(S^n)$,
\item $\displaystyle \max_{M^n\times\{t_0\}}H\leq \Theta R^{-1}$ and
\item $\displaystyle\min_{M^n\times\{t_0\}}\frac{\kappa_1+\dots+\kappa_{m+1}}{H}\geq \alpha$.
\een
Then 
\ba\label{eq:mconvexity}
\lb\kappa_n-\frac{1}{n-m}H\rb(x,t)\leq \varepsilon H(x,t)+C_\varepsilon R^{-1}
\ea
for all $(x,t)\in M^n\times[t_0,T)$.
\end{cor}
\begin{proof}
Since the $(m+1)$-convexity condition describes a convex, $O(n)$-invariant cone in $\mathcal{S}(n)$, we need only observe that each of the cylindrical points admitted by the $(m+1)$-convexity condition (and hence their convex hull) is contained in the convex cone
\bann
\Lambda:=\lcb W\in \mathcal{S}(n): W\leq \frac{1}{n-m}\tr(W)\mathrm{I}\rcb\,.
\eann
\end{proof}

This estimate is also asymptotically sharp, since the left hand side of \eqref{eq:mconvexity} vanishes on the shrinking cylinder $\R^m\times S^{n-m}_{\sqrt{2(n-m)(1-t)}}$. Moreover, it is not implied by the cylindrical estimates \eqref{eq:mcylindrical} (geometrically, the cone which gives rise to the $m$-cylindrical estimate \eqref{eq:mcylindrical} is the round cone whose axis is the umbilic ray and whose boundary contains the $m$-cylindrical points).

Obtaining a sharp estimate for $\kappa_n$ is a key step in our proof of a sharp estimate for the \emph{inscribed curvature}, which we shall now describe. %A further application is the generalization of both estimates to flows by concave functions of curvature, which will be described in a forthcoming paper with Lynch \cite{}.

\subsection{The inscribed curvature}

We turn our attention now to embedded hypersurfaces. Let $M^n=\pd \Omega\subset \R^{n+1}$ be a properly embedded hypersurface bounding a precompact open set $\Omega\subset\R^{n+1}$ and equip $M^n$ with its outward pointing unit normal. Then the \emph{inscribed curvature} $\overline k(x)$ of a point $x\in M^n$ is defined as the curvature of the boundary of the largest ball which is contained in $\Omega$ and has first order contact with $M^n$ at $x$ \cite{ALM13}. A straightforward calculation \cite[Proposition 2]{An12} reveals that
\ba\label{eq:inscribedk}
\overline k(x)=\sup_{y\in M^n\setminus\{x\}}k(x,y)\,,
\ea
where
\bann
k(x,y):=\frac{2\inner{x-y}{\nu(x)}_{\R^{n+1}}}{\norm{x-y}^2_{\R^{n+1}}}\,.
\eann
Similarly, one can define the \emph{exscribed curvature} $\underline k(x)$ at $x$ as the (signed) boundary curvature of the largest ball, halfspace or ball compliment having exterior contact at $x$. In that case, one observes
\ba\label{eq:exscribedk}
\underline k(x)=\inf_{y\in M\setminus\{x\}}k(x,y)\,.
\ea
Note that reversing the orientation of the hypersurface\footnote{For a mean convex hypersurface, we will always define $\overline k$ and $\underline k$ with respect to the normal whose mean curvature is non-negative. This agrees with the outward pointing normal if $M^n$ is connected.} interchanges $\overline k$ and $\underline k$. Observing that either the supremum in \eqref{eq:inscribedk} is attained, or else $\overline k(x)=\limsup_{y\to x}k(x,y)=\sup_{y\in T_xM\setminus\{0\}}A_x(y,y)/g_x(y,y)$, allows one to obtain derivative identities (in, say, the viscosity sense) for $\overline k$ (and similarly for $\underline k$) by analysing the smooth `two-point functions' $k(x,y)$ and $A_x(y,y)/g_x(y,y)$ \cite{An12,ALM13,NC2}. In particular, along a solution of mean curvature flow, we obtain
\ba\label{eq:evolveinscribedk}
(\pd_t-\Delta)\overline k \leq |A|^2\overline k
\ea
and
\ba\label{eq:evolveexscribedk}
(\pd_t-\Delta)\underline k \geq |A|^2\underline k\,.
\ea
Since $H$ solves \eqref{eq:evolveH}, a simple application of the maximum principle reveals that mean convex solutions of mean curvature flow are \emph{interior} (resp. \emph{exterior}) \emph{non-collapsing}: $\overline k$ (resp. $\underline k$) can be compared from above (resp. below) by $H$ uniformly in time. %By keeping track of some discarded gradient terms in the derivation of \eqref{eq:evolvek},
For mean curvature flow of convex hypersurfaces, a straightforward blow-up argument shows that these  ratios become optimal at a singularity \cite{NC2}, yielding a rather straightforward proof of the theorems of Huisken \cite{Hu84} and Gage--Hamilton \cite{GaHa86} on the convergence of convex solutions of mean curvature flow to round points. Moreover, Brendle \cite{Br15} was able to prove, using a Stampacchia iteration argument similar to those described above, that this is also the case for mean convex mean curvature flow (see also \cite{HaKl15}). Precisely, Brendle showed that for any $\varepsilon>0$ there is a constant $C_\varepsilon$, which depends only on $\varepsilon$ and the initial data, such that
\ba\label{eq:inscribed1}
\overline k-H\leq \varepsilon H+C_\varepsilon\,.
\ea
and
\ba\label{eq:exscribed1}
\underline k\geq -\varepsilon H-C_\varepsilon\,.
\ea
These estimates are sharp due to the fact that $\overline k\equiv H$ and $\underline k\equiv 0$ hold identically on a shrinking cylinder $\R^{n-1}\times S^1_{\sqrt{2(n-1)(1-t)}}$. %Similar arguments yield an `exterior' non-collapsing estimate, which can be improved to a sharp lower bound for $\underline k/H$.

This estimate was improved for $(m+1)$-convex mean curvature flow in \cite{La15} using a blow-up argument and the new compactness results of Haslhofer--Kleiner \cite{HK1}. In section \S \ref{sec:inscribed} we will show that this estimate also follows from a Stampacchia iteration argument.

\begin{thm}[Inscribed curvature pinching. Cf. \cite{La15}]\label{thm:inscribed}
Given a dimension $n\geq 2$, a curvature scale $\Theta<\infty$, a pinching constant $\alpha>0$, a collapsing constant $\varLambda<\infty$ and any $\varepsilon>0$, there is a constant $C_\varepsilon<\infty$ with the following property: Let $X:M^n\times[t_0,T)\to\R^{n+1}$ be a compact solution of \mcf satisfying, for some $m\in\{0,\dots,n-2\}$,
\ben
\item $\displaystyle \lb\frac{\mu_{t_0}(M^n)}{\sigma_n}\rb^{\frac{1}{n}}+\sqrt{2n(T-t_0)}\leq 2R$, where $\sigma_n:=\mathrm{Area}(S^n)$,
\item $\displaystyle \max_{M^n\times\{t_0\}}H\leq \Theta R^{-1}$,
\item $\displaystyle\min_{M^n\times\{t_0\}}\frac{\kappa_1+\dots+\kappa_{m+1}}{H}\geq \alpha$ and
\item $\displaystyle\max_{M^n\times\{t_0\}}\frac{\overline k}{H}\leq \varLambda$.
\een
Then
\bann
\lb\overline k-\frac{1}{n-m}H\rb(x,t)\leq \varepsilon H(x,t)+C_\varepsilon R^{-1}
\eann
for all $(x,t)\in M^n\times[t_0,T)$.
\end{thm}

We note that in \cite{Br15}, the $m=n-1$ case of \eqref{eq:mconvexity} (which follows from the convexity estimate \eqref{eq:convexity}) is used to reduce to the `interior' case that the supremum in \eqref{eq:inscribedk} is attained (recall that $\overline k=\kappa_n$ otherwise). The $m$-convexity estimates \eqref{eq:mconvexity} play this role %\footnote{In fact, unlike Brendle's proof, we need only make use of \emph{upper} curvature pinching. This observation will be crucial in \cite{LaLy} where we obtain similar estimates for flows by \emph{concave} functions of curvature, which, as we have mentioned, favour upper curvature pinching.}
 in our proof.

One advantage of the Stampacchia iteration argument is that it requires only one-sided non-collapsing (whereas the proof described in \cite{La15} makes use of the techniques of \cite{HK1}, which fundamentally require two-sided non-collapsing). In a separate article%\cite{LaLy}
, with Lynch, we show how to obtain analogous estimates for flows by non-linear functions of curvature, where, in general, only one-sided non-collapsing holds.

\subsection{Ancient solutions}

In the second part of the paper, we consider \emph{ancient solutions} of \eqref{eq:MCF}. These are solutions of \mcf which are defined for time intervals $I$ of the form $(-\infty,T)$ with $T\leq \infty$. For compact $M^n$, we can assume without loss of generality that $T=1$ is the maximal existence time. In principle, this property should be extremely rigid, since diffusion has had an arbitrarily long time to take effect. Indeed, when $n=1$, the only compact, convex, embedded ancient solutions are shrinking circles and Angenent ovals \cite{DHS10}. For $n\geq 2$, an analogous classification remains open; however, some recent breakthroughs have been made. For instance, given any $\alpha>0$, the shrinking sphere is the only $\alpha$-\emph{non-collapsing} (that is, $H>0$ and $-\alpha H\leq \underline k\leq\overline k\leq \alpha H$) ancient solution which is either uniformly convex or of type-I curvature growth (that is, $\limsup_{t\to-\infty}\sqrt{1-t}\max_{M^n\times\{t\}}H<\infty$) \cite{HaHe}. %Note that non-collapsing rules out the Angenent oval, and any other solution which has points where $\overline k/H$ blows up. 
In fact, the same statement is true for ancient solutions in dimensions $n\geq 2$ when the non-collapsing condition is replaced by \emph{convexity} \cite{HuSi15}. The proof of the latter result is based on a clever modification of the proof of Huisken's umbilic estimate. The same idea applies to the cylindrical estimates \cite{HuSi15}, so that convex, uniformly $(m+1)$-convex ancient solutions satisfy
\ba\label{eq:mcylindricalancient}
|A|^2-\frac{1}{n-m}H^2\leq 0\,.
\ea
Moreover, arguing via the strong maximum principle, it is shown that strict inequality holds unless $m=1$ (in which case the solution is necessarily the shrinking sphere).

Adapting the argument of Huisken--Sinestrari, we are able to obtain a sharp pinching estimate for ancient solutions, as long as the solution has \emph{bounded rescaled volume}; that is,
\ba\label{eq:brv}
\limsup_{t\to-\infty}\frac{1}{(-t)^{n+1}}\int_t^0\hspace{-3mm}\int_{M^n}H(\cdot,s)\,d\mu_s\,ds<\infty\,.
\ea
Note that, when the evolving hypersurfaces $M^n_t$ are mean convex and bound precompact regions $\Omega_t\subset \R^{n+1}$,
\bann
|\Omega_t|=|\Omega_0|+\int_t^0\hspace{-3mm}\int_{M^n}H(\cdot,s)\,d\mu_s\,ds\,.
\eann
%We can take this as a definition of $|\Omega_t|-|\Omega_0|$ for flows of immersed (one-sided) hypersurfaces.

\begin{thm}\label{thm:pinchingancient}
Fix a dimension $n\in \N\setminus\{1\}$ and a pinching condition $\Gamma\subset \mathcal{S}(n)$ and denote by $\Lambda$ the convex hull of the cylindrical points lying in $\Gamma$. Let $X:M^n\times(-\infty,1)\to\R^{n+1}$ be a compact ancient solution of \mcf with bounded rescaled volume. Suppose, in addition, that the solution is uniformly pinched, in the sense that
\bann
A_{(x,t)}\in \overline\Gamma_0\quad\text{for all}\quad (x,t)\in M^n\times(-\infty,0]
\eann
for some pinching condition $\Gamma_0$ satisfying $\overline\Gamma_0\setminus\{0\}\subset \Gamma$. Then
\bann
A_{(x,t)}\in \Lambda\quad\text{for all}\quad (x,t)\in M^n\times(-\infty,1)\,.
\eann
Moreover, if $M^n$ is connected, then $A_{(x,t)}\in \mathrm{int}(\Lambda)$ for all $(x,t)\in M^n\times(-\infty,1)$ unless $\Lambda$ is the umbilic ray and $M^n_t$ the shrinking sphere $S^n_{\sqrt{2n(1-t)}}$.
\end{thm}

We note that it was left open in \cite{HuSi15} whether or not the techniques apply to the convexity estimate. The difficulty appears to arise in the induction step in the proof of the convexity estimate in \cite{HuSi99b}: In order to start the $(k+1)$-st step, we require uniform pinching $H_{k}\geq \alpha H^k$ for some $\alpha>0$, where $H_k$ is the $k$-th mean curvature; however, the conclusion of the $k$-th step only yields strict pinching $H_k>0$. Under the assumption of bounded rescaled volume, the desired estimate is an immediate corollary of Theorem \ref{thm:pinchingancient}.

\begin{cor}\label{cor:convexityancient}
Fix $n\in \N\setminus\{1\}$ and let $X:M^n\times(-\infty,1)\to\R^{n+1}$ be a compact, mean convex ancient solution of \mcf with bounded rescaled volume. Suppose, in addition, that
\bann
\liminf_{t\to-\infty}\frac{\kappa_1}{H}>-\infty\,.
\eann
Then the solution is strictly convex for all $t\in(-\infty,1)$.
\end{cor}

The conditional bounded rescaled volume appears to be quite mild. Indeed, it holds automatically if $H$ is uniformly bounded in $L^n$ or $L^\infty$ for $t<0$ (see Lemma \ref{lem:weakHbound}). The latter is clearly true for type-I ancient solutions and, moreover, follows for convex ancient solutions from Hamilton's Harnack estimate \cite{HuSi15}. A similar estimate holds for (interior) non-collapsing solutions \cite{ShWa09} (see also \cite{HK1}). In their classification of embedded, closed, convex ancient solutions of the curve shortening flow, Daskalopoulos, Hamilton and \u{S}e\u{s}um show that bounds for the curvature in $L^1$ and $L^\infty$ are sufficient (and necessary) to deduce that a closed, embedded ancient solution is convex \cite{DHS10}. A sup-bound for the speed was assumed in the recent, very general, classification of convex ancient solutions of curvature flows in the sphere \cite{BIS} when the corresponding flow does not admit an appropriate Harnack estimate\footnote{It is tempting to conjecture that all ancient solutions should have bounded mean curvature as $t\to-\infty$; however, recent numerical evidence suggests that this is false \cite{Angenent}.}. In any case, this is already sufficient to weaken the convexity assumption in the rigidity result of Huisken and Sinestrari \cite{HuSi15} and (when $n\geq 2$) the non-collapsing assumption in the result of Haslhofer and Hershkovitz \cite{HaHe}.

\begin{cor}\label{cor:shrinkingsphere}
Fix $n\in \N\setminus\{1\}$ and let $X:M^n\times(-\infty,1)\to\R^{n+1}$ be a compact, connected, mean convex, embedded ancient solution of \mcf satisfying
\bann
\liminf_{t\to-\infty}\frac{\kappa_1}{H}>-\infty\,.
\eann
Then the following are equivalent:
\ben
\item $M_t^n$ is the shrinking sphere $S^n_{\sqrt{2n(1-t)}}$
\item $M^n_t$ is uniformly convex:
\[
\liminf_{t\to-\infty}\min_{M^n\times\{t\}}\frac{\kappa_1}{H}>0\,.
\]
\item $M^n_t$ has bounded rescaled diameter:
\[
\limsup_{t\to-\infty}\frac{\diam(M_t^n)}{\sqrt{1-t}}<\infty\,.
\]
\item $M^n_t$ has bounded eccentricity: 
\[
\limsup_{t\to-\infty}\frac{\rho_+(t)}{\rho_-(t)}<\infty\,,
\]
where $\rho_+(t)$ and $\rho_-(t)$ denote, respectively, the circum- and in-radii of $M_t^n$.
\item $M^n_t$ has bounded mean curvature ratios:
\[
\limsup_{t\to-\infty}\frac{\max_{M^n\times\{t\}}H}{\min_{M^n\times\{t\}}H}<\infty\,.
\]
\item $M^n_t$ has type-I curvature growth:
\[
\limsup_{t\to-\infty}\sqrt{1-t}\max_{M^n\times\{t\}}H<\infty\,.
\]
\item $M^n_t$ satisfies a reverse isoperimetric inequality:
\[
\limsup_{t\to-\infty}\frac{\mu_t(M)^{n+1}}{|\Omega_t|^n}<\infty\,.
\]
%\item $M^n_t$ satisfies a reverse Sobolev inequality: 
\een
\end{cor}

As a consequence of the convexity estimate, we also obtain a sharp estimate for the exscribed curvature, so long as the flow is exterior non-collapsing.

\begin{thm}\label{thm:exscribedancient}
Fix $n\in \N$ and let $X:M^n\times(-\infty,1)\to\R^{n+1}$ be a compact, mean convex ancient solution of \mcf with bounded rescaled volume. Suppose, in addition, that
\[
\liminf_{t\to-\infty}\min_{M^n\times\{t\}}\frac{\underline k}{H}>-\infty\,.
\]
Then $M^n_t$ bounds a strictly convex region for all $t\in (-\infty,1)$.
\end{thm}
In particular, any compact, mean convex ancient solution of \mcf with bounded rescaled volume and more than one connected component is exterior collapsing as $t\to-\infty$.

%\begin{cor}
%Let $M^n$ be a compact $n$-manifold, $n\geq 1$, and $X:M^n\times(-\infty,1)\to\R^{n+1}$ a mean convex ancient solution of \mcf with bounded rescaled volume. If, in addition,
%\[
%\liminf_{t\to-\infty}\min_{M^n\times\{t\}}\frac{\underline k}{H}>-\infty\,,
%\]
%then $M^n$ is connected.
%\end{cor}

Another immediate corollary of Theorem \ref{thm:pinchingancient} is a sharp estimate for the largest principal curvature.
\begin{cor}\label{cor:mconvexityancient}
Fix $n\in \N\setminus\{1\}$ and let $X:M^n\times(-\infty,1)\to\R^{n+1}$ be a compact, mean convex ancient solution of \mcf with bounded rescaled volume. Suppose, in addition, that
\bann
\liminf_{t\to-\infty}\min_{M^n\times\{t\}}\frac{\kappa_1+\dots+\kappa_{m+1}}{H}>0
\eann
for some $m\in \{0,\dots,n-2\}$. Then
\bann
\lb\kappa_n-\frac{1}{n-m}H\rb(x,t)\leq 0
\eann
for all $(x,t)\in M^n\times(-\infty,1)$. Moreover, if $M^n$ is connected, then the inequality is strict, unless $m=0$ and $M^n_t$ the shrinking sphere $S^n_{\sqrt{2n(1-t)}}$.
\end{cor}

This allows us to obtain a sharp estimate for the inscribed curvature, so long as the flow is interior non-collapsing.

\begin{thm}\label{thm:inscribedancient}
Fix $n\in \N\setminus\{1\}$ and $m\in \{0,\dots,n-1\}$ and let $X:M^n\times(-\infty,1)\to\R^{n+1}$ be a compact, mean convex ancient solution of \mcf with bounded rescaled volume. Suppose, in addition, that
\[
\liminf_{t\to-\infty}\min_{M\times\{t\}}\frac{\kappa_1+\dots+\kappa_{m+1}}{H}>0
\]
and
\[
\limsup_{t\to-\infty}\max_{M\times\{t\}}\frac{\overline k}{H}<\infty\,.
\]
Then
\bann
\overline k(x,t)-\frac{1}{n-m}H(x,t)\leq 0
\eann
for all $(x,t)\in M^n\times(-\infty,1)$. Moreover, if $M^n$ is connected, then the inequality is strict, unless $m=0$ and $M^n_t$ the shrinking sphere $S^n_{\sqrt{2n(1-t)}}$.
\end{thm}

\section*{Acknowledgements}

This work has been discussed in the geometric analysis research seminar directed by Klaus Ecker at the Freie Universit\"at Berlin. I am grateful to the members of the geometric analysis group and to the students who attended this seminar for many useful comments. I am particularly indebted to Stephen Lynch for providing helpful comments on a draft of this paper. I am also grateful to Carlo Sinestrari for interesting discussions about ancient solutions of the mean curvature flow. Finally, I wish to acknowledge the financial support of the Alexander von Humboldt Foundation, whose fellowship made this work possible.

\section{Preliminaries}\label{sec:prelims}

In this section, we collect some background results which are needed for the proofs of the main theorems but may have a wider range of applicability. Particular results of interest are Propositions \ref{prop:evolveG}, \ref{prop:splitting}, \ref{prop:Poincare} and \ref{prop:kSimons}. 

We will begin with some evolution equations. So let $X:M^n\times I\to\R^{n+1}$ be a smooth  solution of \mcf for some compact manifold $M^n$. Then the family $\mu_t:=\mu(\cdot,t)$, $t\in I$, of measures induced by the immersions $X_t$ satisfy
\ba\label{eq:evolvemu}
\frac{d}{dt}\int \eta\, d\mu=\int \lb\pd_t\eta- \eta H^2\rb d\mu
\ea
for any $\eta\in C^\infty(M^n\times I)$, which is nothing more than the first variation formula for the area. In fact, \eqref{eq:evolvemu} holds for almost every $t$ for test functions which are only $L^1$ in space and $W^{1,1}$ in time.

\begin{comment}
For flows in $\R^3$ we can bound the maximal time by combining \eqref{eq:evolvemu} with the Alexandrov maximum principle.

\begin{lem}
Let $X:M^2\times[t,T)\to\R^{3}$ be a compact solution of \eqref{eq:MCF}. Then
\bann
\mu_t(M^2)\geq 4A(T-t)\,,
\eann
where $A:=\mathrm{Area}(S^2)$.
\end{lem}
\begin{proof}
Let $M_+^2$ be the contact set of $M^n$. Then, using \eqref{eq:evolvemu}, the inequality of arthmetic and geometric means and the area formula, we obtain
\bann
-\frac{d}{dt}\mu(M^2)=\int H^2\,d\mu\geq \int_{M_+^2}H^2\geq 4\int_{M_+^2}K=4A\,,
\eann
where $K=\det(D\nu)$ is the Gauss curvature of $M^2$ and $A:=\mathrm{Area}(S^2)$. Integrating yields the claim.
\end{proof}
\end{comment}

Next, we consider functions of curvature. First, by a simple computation making use of \eqref{eq:evolveA}, we find
\ba\label{eq:evolvenormA}
(\pd_t-\Delta)|A|=|A|^2|A|-\frac{1}{2|A|^3}|A\otimes\cd A-\cd A\otimes A|^2
\ea
wherever $A\neq 0$. The gradient term on the right hand side will prove useful.

%\bann
%\pd_t|A|=\frac{1}{|A|}\inner{\cd_tA}{A}
%\eann

%\bann
%\cd_k|A|=\frac{1}{|A|}\inner{\cd_kA}{A}
%\eann

%\bann
%\Delta|A|={}&\frac{1}{|A|}\inner{\Delta A}{A}-\frac{1}{|A|^3}\inner{\cd_kA}{A}\inner{\cd_kA}{A}+\frac{1}{|A|}|\cd A|^2\\
%={}&\frac{1}{|A|^3}\lb |A|^2|\cd A|^2-\inner{\cd_kA}{A}\inner{\cd_kA}{A}\rb\\
%={}&\frac{1}{2|A|^3}|A\otimes\cd A-\cd A\otimes A|^2
%\eann

\begin{lem}[Cf. {\cite[Lemma 2.3]{Hu84}}]\label{lem:gradterm}
Given a dimension $n\in\N\setminus\{1\}$ and an open, $O(n)$-invariant cone $\Gamma_0\subset \mathcal{S}(n)$ whose closure does not contain the cylindrical point $\mathrm{diag}(0,\dots,0,1)$, there is a constant $\gamma>0$ with the following property: Given a smooth, strictly mean convex immersion $X:M^n\to\R^{n+1}$,
\bann
|A\otimes\cd A-\cd A\otimes A|^2\geq \gamma|A|^2|\cd A|^2
\eann
on the set $M_{0}^n:=\{x\in M^n:A_x\in \overline\Gamma_0\}$.
\end{lem}
\begin{proof}
Fix any $x\in M^n_0$ at which $|A||\cd A|\neq 0$ and rescale so that $|A||\cd A|=1$. By compactness of the set $\{(W,T)\in \overline \Gamma_0\times (\R^n\odot \R^n\odot\R^n):|W||T|=1\}$, where $\odot$ denotes the symmetric tensor product, it suffices to prove that
\bann
|A\otimes\cd A-\cd A\otimes A|^2>0
\eann
at $x$. Suppose, to the contrary, that
\bann
A\otimes\cd A=\cd A\otimes A\,.
\eann
In a principal frame, this becomes, after applying the Codazzi identity,
\ba\label{eq:gradterm1}
\kappa_p\delta_{pq}\cd_kA_{ij}=\kappa_i\delta_{ij}\cd_kA_{pq}
\ea
for each $p$, $q$, $k$, $i$ and $j$. By hypothesis, $\kappa_n>0$. Fix $k$, $i$ and $j$ so that $\cd_kA_{ij}\neq 0$. Then
\ba\label{eq:gradterm2}
\kappa_n\cd_kA_{ij}=\kappa_i\delta_{ij}\cd_kA_{nn}\,,
\ea
so that, in particular, $i=j$. By the Codazzi identity, the same argument 
%\bann
%\kappa_n\cd_kA_{ij}=\kappa_n\cd_iA_{kj}\,,
%\eann
%so that, arguing similarly,
implies $k=j$. Thus, $\cd_kA_{ij}$ is non-zero only if $k=i=j$. Returning to \eqref{eq:gradterm2}, we find
\bann
\kappa_n\cd_kA_{kk}=\kappa_k\cd_kA_{nn}
\eann
and conclude that $k=n$. That is, $\kappa_n\cd_kA_{ij}$ is non-zero only if $k=i=j=n$. On the other hand, for any $i\neq n$, \eqref{eq:gradterm1} yields
\bann
\kappa_i\cd_nA_{nn}=\kappa_n\cd_nA_{ii}=0\,.
\eann
It follows that $\kappa_i=0$ unless $i=n$. But this contradicts the hypothesis $A_x\in \overline \Gamma_0$.
\end{proof}

Next, we will consider the function which gives the distance of the curvature $A$ to the boundary of a pinching set $\Gamma\subset \mathcal{S}(n)$, but first we need to recall some facts from convex geometry: Given a convex subset $C\subset E$ of a finite dimensional normed linear space $E$, we recall that the signed distance to the boundary of $E$ is given by
\ba\label{eq:distance}
d_{C}(x)=\inf_{\ell\in \mathrm{S}C}\ell(x)\,,
\ea
where $\mathrm{S}C$ denotes the set of \emph{supporting affine functionals} for $C$; that is, the set of affine linear maps $\ell:E\to \R$ satisfying $\norm{D\ell}=1$ and $\ell(x)\geq 0$ for all $x\in C$ with equality at some $x_0\in \pd C$. We will also say that $\ell$ \emph{supports $C$ at $x_0$} and denote by $\mathrm{S}_{x_0}C$ the set of supporting affine functionals which support $C$ at $x_0$. Note that $d_C(x)$ is the distance from $x$ to $E\setminus C$ if $x\in C$ and the negative of the distance from $x$ to $C$ if $x\in E\setminus C$. Note also that, in case $C$ is a cone, $\mathrm{S}C\subset E^\ast$; that is, the supporting affine functionals are linear functionals. Moreover, by the Hahn--Banach Theorem, the infimum in \eqref{eq:distance} is always attained by some $\ell\in \mathrm{S}C$.

Finally, observe that a convex, $O(n)$-invariant cone $\Gamma\subset \mathrm{S}(n)$ defines a convex, symmetric\footnote{That is, invariant under permutation of components.} cone $\gamma\subset \R^n$ (and vice versa) via the rule
\[
z\in \gamma\quad\iff\quad o(\diag(z))\subset \Gamma\,,
\]
where $o$ denotes the orbit under the $O(n)$ action, and
\bann
d_{\gamma}(z)=d_{\Gamma}(Z)
\eann
for any $z\in \gamma$ and $Z\in o(\diag(z))$.

\begin{prop}\label{prop:evolveG}
Let $X:M^n\times I\to\R^{n+1}$ be a solution of mean curvature flow. Given any closed, convex, $O(n)$-invariant cone $\Gamma\subset\mathcal{S}(n)$, the function $G:\M^n\times I\to\R$ defined (with respect to some, and hence any, orthonormal basis) by
\bann
G(x,t):=d_\Gamma(A_{(x,t)})
\eann
satisfies
\ba\label{eq:evolveG}
(\pd_t-\Delta)G\geq |A|^2G%+Q(\cd A)
\ea
in both the viscosity and the distributional sense.%, where, with respect to a principal frame,
%\bann
%Q(\cd A):=\sup_{L\in \mathrm{S}\Gamma}\sum_{k=1}^n\sum_{\kappa_i<\kappa_j}\frac{\ell^i-\ell^j}{\kappa_j-\kappa_i}(\cd_kA_{ij})^2\,,
%\eann
%where $\ell^i:=L^{ii}$.
\end{prop}
%\begin{rem}
%Note that $G$ is well-defined by the $O(n)$-invariance of $\Gamma$.
%\end{rem}
\begin{proof}%[Proof of Proposition \ref{prop:evolveG}]
We first show that the inequality holds in the viscosity sense. Fix $(x_0,t_0)\in M^n\times I$ and let $\varphi\in C^\infty(B_r(x_0,t_0)\times(t_0-r^2,t_0])$, $r>0$, be any smooth lower support function for $G$ at $(x_0,t_0)$; that is,
\bann
\varphi\leq G \quad\text{on}\quad B_r(x_0,t_0)\times (t_0-r^2,t_0]\quad\text{and}\quad \varphi(x_0,t_0)=G(x_0,t_0)\,.
\eann
Here $B_r(x_0,t_0):=\{x\in M^n:d_{t_0}(x,x_0)<r\}$ denotes the time $t_0$ metric ball of radius $r$ centred at $x_0$.

We need to show that $\varphi$ satisfies
\bann
(\pd_t-\Delta)\varphi\geq |A|^2\varphi%+\sum_{k=1}^n\sum_{\kappa_i<\kappa_j}\frac{\ell^i-\ell^j}{\kappa_j-\kappa_i}(\cd_kA_{ij})^2
\eann
at the point $(x_0,t_0)$. Fix an orthonormal basis $\{e^0_i\}_{i=1}^n$ at $(x_0,t_0)$ and let $L_0\in \mathrm{S}\Gamma$ be a supporting affine functional satisfying
\bann
G=L_0^{ij}A_{ij}\quad\text{at}\quad (x_0,t_0)\,.
\eann
Having fixed an orthonormal basis at $(x_0,t_0)$, we can consider $L_0$ as a symmetric bilinear form acting on $T_{x_0}M$. We smoothly extend $L_0$ to a symmetric bilinear form $L$ defined in a neighbourhood of $(x_0,t_0)$ by setting $L:=L_0^{ij}e_i\otimes e_j$, where the orthonormal frame $\{e_i\}_{i=1}^n$ is formed by solving
\bann
\overline\cd e_i\equiv 0\quad\text{and}\quad\cd_te_i\equiv 0\quad\text{with}\quad e_i(x_0,t_0)=e_i^0
\eann
for any metric connection $\overline\cd$. Set $\phi(x,t):=L(A_{(x,t)})$. Choosing $r$ smaller if needed, we can arrange that $\phi$ is defined on $B_r(x_0,t_0)\times (t_0-r^2,t_0]$. Moreover, since the basis $\{e_i\}_{i=1}^n$ remains orthonormal and $L_0\in \mathrm{S}\Gamma$, it is a consequence of the definition \eqref{eq:distance} that
\bann
\phi\geq G \quad\text{on} \quad B_r(x_0,t_0)\times(t_0-r^2,t_0]\,. %\quad\text{and}\quad \phi(x_0,t_0)=G(x_0,t_0)\,;
\eann
Thus, $\phi$ is a smooth \emph{upper} support for $G$ at $(x_0,t_0)$. Since $\phi$ is smooth, we can compute
\ba\label{eq:metricconnection}
(\pd_t-\Delta)\varphi\geq{}&(\pd_t-\Delta)\phi\nonumber\\
={}&(\pd_t-\Delta)(L(A))\nonumber\\
={}&L((\cd_t-\Delta)A)-g^{kl}\lb 2\cd_kL(\cd_lA)+\cd_k\cd_lL(A)\rb\nonumber\\
%={}&-|A|^2L(A)+g^{kl}\lb 2\cd_kL(\cd_lA)+\cd_k\cd_lL(A)\rb\nonumber\\
={}&|A|^2\varphi-g^{kl}\lb 2\cd_kL(\cd_lA)+\cd_k\cd_lL(A)\rb
\ea
at $(x_0,t_0)$. Choosing $\overline \cd:=\cd$ yields the claim (later, we will choose $\overline \cd$ more carefully).

To see that the inequality is satisfied in the distributional sense, it now suffices, by Alexandroff's Theorem \cite{Al39}, to show that $G$ is locally quasi-concave \cite[Chapter 6]{EvansGariepy}. To prove this, fix $(x_0,t_0)$ and let $\phi\in C^\infty(B_r(x_0)\times(t_0-r^2,t_0])$ be the upper support for $G$ at $(x_0,t_0)$ constructed above. Setting $u(s):=G\circ\gamma$ for any unit length geodesic $\gamma:I\to M^n$ with $(\gamma(0),\gamma'(0))=(x_0,v)$, we find
\bann
u(s)-u(0)\leq{}&\phi(\gamma(s))-\phi(x_0)\\
={}&s\nabla_v\phi(x_0)+\frac{1}{2}s^2\nabla^2_{v,v}\phi(x_0)+o(s^2)\,.
\eann
Now set $u_{\lambda}(s):=u(s)-\lambda s^2$, where $\lambda:=\sup_{B_{r/2}(x_0,t_0)}|\nabla^2\phi|$. Then
\bann
u_\lambda(s)-u_\lambda(0)\leq{}&s\nabla_v\phi(x_0)+\frac{1}{2}s^2\nabla^2_{v,v}\phi(x_0)-\lambda s^2+o(s^2)\\
\leq{}&s\nabla_v\phi(x_0)-\frac{\lambda}{2} s^2+o(s^2)\,.
\eann
For $s$ sufficiently small, we obtain
\bann
u_\lambda(s)\leq{}&u_\lambda(0)+s\nabla_v\phi(x_0)\,.
\eann
Thus, at each point, we have found a supporting line lying locally above the graph of $u_\lambda$. This proves the claim.
\end{proof}

The following well-known `tensor maximum principle' is an immediate corollary.

\begin{cor}
Let $X:M^n\times [0,T)\to\R^{n+1}$ be a solution of mean curvature flow and $\Gamma\subset \mathcal{S}(n)$ a closed, convex, $O(n)$-invariant cone. Suppose that $A_{(x,0)}\in \Gamma$ for all $x\in M^n$. Then $A_{(x,t)}\in \Gamma$ for all $(x,t)\in M^n\times [0,T)$.
\end{cor}

A more careful analysis yields a stronger statement for the cones
\bann
\Lambda_m:=\mathrm{Conv}\{\mathrm{Cyl}_m\}\,,
\eann
where $\mathrm{Conv}$ denotes the convex hull in $\mathcal{S}(n)$ and $\mathrm{Cyl}_m$ denotes the set of \emph{$m$-cylindrical points}; that is, the points $W\in \mathcal{S}(n)$ with a null eigenvalue of multiplicity $m$ and a positive eigenvalue of multiplicity $n-m$.
\begin{prop}\label{prop:splitting}
Let $X:M^n\times [0,T)\to\R^{n+1}$, $n\geq 2$, be a compact, connected solution of mean curvature flow and $\Lambda_m\subset \mathcal{S}(n)$ the convex hull of the $m$-cylindrical points for some $m\in\{0,\dots,n-1\}$. Suppose that $A_{(x,0)}\in \Lambda_m$ for all $x\in M^n$. Then either $A_{(x,t)}\in \mathrm{int}(\Lambda_m)$ for all $(x,t)\in M^n\times (0,T)$ or $m=0$ and $M^n_t$ is a shrinking sphere.
\end{prop}
\begin{proof}
First, we make a more careful choice of the metric connection in \eqref{eq:metricconnection} in order to obtain a good gradient term (cf. \cite[Theorem 3.2]{An07}). Given any $C\in \Gamma(M\times [0,T),T^\ast M\otimes T^\ast M\otimes TM)$ satisfying
\ba\label{eq:antisymmetry}
g(C(w,u),v)+g(u,C(w,v))=0\quad\text{for all}\quad u,v,w\in TM
\ea
we can define a metric connection $\overline \cd$ on $TM$ (and the entire tensor algebra using the Leibniz rule) via
\[
\overline \cd_uv:=\cd_uv+C(u,v)\,.
\]
Recalling \eqref{eq:metricconnection}, we want to estimate the gradient term
\bann
Q_C(\cd A):={}&-g^{kl}\big( 2\cd_kL(\cd_lA)+\cd_k\cd_lL(A)\big)
\eann
at the point $(x_0,t_0)$, where, with respect to the $\overline \nabla$-parallel frame $\{e_i\}_{i=1}^n$, $L=L_0^{ij}e_i\otimes e_j$ for some supporting affine functional $L_0\in\mathrm{S}\Lambda_m$ satisfying $G(x_0,t_0)=L_0(A_{(x_0,t_0)})$. So we need to compute
\bann
\cd_kL=L_0^{ij}\lb C_{ki}{}^pe_p\otimes e_j+C_{kj}{}^qe_i\otimes e_q\rb
\eann
and
\bann
\cd_k\cd_lL={}&L_0^{ij}\lb\cd_kC_{li}{}^pe_p\otimes e_j+\cd_kC_{lj}{}^qe_i\otimes e_q\right.\\
{}&+C_{li}{}^pC_{kp}{}^re_r\otimes e_j+C_{li}{}^pC_{kj}{}^qe_p\otimes e_q\\
{}&\left.+C_{lj}{}^qC_{ki}{}^pe_p\otimes e_q+C_{lj}{}^qC_{kq}{}^re_i\otimes e_r\rb\,.
\eann
To simplify things, we can arrange that $\{e^0_i\}_{i=1}^n$ is a principal frame and
\bann
L_0=\diag(\ell_0)
\eann
for some $\ell_0\in \mathrm{S}_{\vec\kappa(x_0,t_0)}\lambda_m$, where $\lambda_m$ is the convex, symmetric cone in $\R^n$ corresponding to $\Lambda_m$ and $\vec\kappa$ denotes the eigenvalue $n$-tuple $(\kappa_1,\dots,\kappa_n)$. Then, making use of the antisymmetry \eqref{eq:antisymmetry}, we find
\bann
Q_C(\cd A)={}&-2\sum_{k,i,p=1}^n\ell_0^iC_{kip}\lb2\cd_{k}A_{ip}+C_{kip}(\kappa_p-\kappa_i)\rb\\
={}&-2\sum_{k=1}^n\sum_{i<p}\lb\ell_0^i-\ell_0^p\rb C_{kip}\lb2\cd_{k}A_{ip}+C_{kip}(\kappa_p-\kappa_i)\rb\
\eann
at the point $(x_0,t_0)$.

Noting that $\cd_kA_{ij}=0$ whenever $\kappa_i=\kappa_j$, we can rewrite this as
\bann
Q_C&(\cd A)=\\
{}&2\sum_{k=1}^n\sum_{\kappa_i<\kappa_p}\lb\ell_0^i-\ell_0^p\rb\lb\kappa_p-\kappa_i\rb \lb\frac{(\cd_kA_{ip})^2}{(\kappa_p-\kappa_i)^2}-\lsb C_{kip}+\frac{\cd_{k}A_{ip}}{\kappa_p-\kappa_i}\rsb^2\rb\,.
\eann
If we choose $C$ so that
\[
(\kappa_j-\kappa_i)C_{kij}=-\cd_{k}A_{ij}
\]
at $(x_0,t_0)$ and set
\bann
Q(\cd A):=2\sup_{\ell\in \mathrm{S}_{\vec\kappa}\lambda_m}\sum_{k=1}^n\sum_{\kappa_i<\kappa_p}\frac{\ell^i-\ell^p}{\kappa_p-\kappa_i}(\cd_kA_{ip})^2\,,
\eann
then we have proved that $G$ satisfies
\bann
(\pd_t-\Delta)G\geq |A|^2G+Q(\cd A)
\eann
in the viscosity sense.

Note that the gradient term is non-negative.
\begin{claim}
Let $\gamma\subset \R^n$ be a convex, symmetric cone. Then for each $w\in \gamma$ and each $\ell\in \mathrm{S}\gamma$ such that $d_\gamma(w)=\ell(w)$,
\[
(\ell^i-\ell^j)(w_j-w_i)\geq 0\quad\text{for each}\quad i,\, j\,.
\]
\end{claim}
\begin{proof}
Fix $i,j\in\{1,\dots,n\}$. By symmetry of $\gamma$, we have $\widehat w\in \gamma$, where $\widehat w$ is obtained from $w$ by interchanging its $i$-th and $j$-th component; that is,
\[
\widehat w:=w-(w_i-w_j)(e_i-e_j)\,.
\]
%Then
%\[
%\sum_{i=k}^n\ell^k\widehat w_k=\sum_{k=1}^n\ell^kw_k-(\ell^i-\ell^j)(w_i-w_j)\,.
%\]
Since $\ell(w)=d_{\gamma}(w)=d_{\gamma}(\widehat w)=\inf_{\widehat\ell\in \mathrm{S}\gamma}\widehat \ell(\widehat w)$, we obtain
\[
(\ell^j-\ell^i)(w_i-w_j)=\sum_{k=1}^n\ell^k(\widehat w_k-w_k)\geq d_{\gamma}(\widehat w)-d_{\gamma}(w)=0\,.
\]
\end{proof}

So suppose now that $G$ reaches zero at an interior point $(x_0,t_0)$. Then, by the strong maximum principle \cite{DL04}, $G\equiv 0$. Hence $G$ is smooth and satisfies
\bann
\cd G\equiv 0\quad\text{and}\quad Q(\cd A)\equiv 0\,.
\eann
It follows that
\bann
\sum_{i=1}^n\ell^i\cd_kA_{ii}=0\quad\text{for each} \quad k
\eann
and
\bann
\ell^i\cd_{k}A_{ij}=\ell^j\cd_{k}A_{ij} \quad\text{for each}\quad k,i,j
\eann
at each point $(x,t)$ for every supporting $\ell\in \mathrm{S}_{\vec\kappa(x,t)}\gamma$. Putting these together, we find
\bann
\ell^k\cd_kH=\sum_{i=1}^n\ell^k\cd_kA_{ii}=\sum_{i=1}^n\ell^k\cd_iA_{ki} =\sum_{i=1}^n\ell^i\cd_iA_{ki}=\sum_{i=1}^n\ell^i\cd_kA_{ii}=0
\eann
for each $k$. If $\cd H\equiv 0$ then, by compactness of $M$, the time slices must be round spheres \cite{Al39}, which implies the claim; otherwise, there is a point at which $\ell^k=0$ for some $k$. That is, $\ell\cdot e_k=0$. We claim that $\kappa_1=0$ at such a point.

\begin{claim}
Let $\ell\in \mathrm{S}_w\lambda_m$ support $\lambda_m$ at $w\in \pd\lambda_m$. Then either $\ell^k\neq 0$ for every $k$ or there is an $l$ such that $w_l=0$.
\end{claim}
\begin{proof}
The claim is evident for $m=n-1$, since in that case $\lambda_m$ is the non-negative cone $\overline \gamma_+=\{w\in \R^n:\min_{1\leq i\leq n}w_i\geq 0\}$. So suppose that $m\leq n-2$. Note that $\lambda_m$ is the convex hull of the cylindrical points $\{(w^m_{\sigma(1)},\dots,w^m_{\sigma(n)}):\sigma\in P_n\}$, where $P_n$ denotes the set of permutations of $\{1,\dots,n\}$ and
\bann
w^m:=(\underbrace{0,\dots,0}_{m-\text{times}},1,\dots,1)\,.
\eann
Thus,
\ba\label{eq:msupport0}
\sum_{i=1}^n\ell^iw^m_{\sigma(i)}\geq 0
\ea
for each $\sigma\in P_n$. If $\ell^k=0$ for some $k\in\{1,\dots,n\}$, then applying \eqref{eq:msupport0} to each $\sigma$ satisfying $\sigma(k)\in \{m+1,\dots,n\}$ (i.e. by putting one of the $1$'s in the $\sigma(k)$-th position), we find that the sum of any $n-(m+1)$ of the components of $\ell$ is non-negative; that is,
\ba\label{eq:msupport1}
\sum_{m+2}^n\ell^{\sigma(i)}\geq 0
\ea
for each $\sigma\in P_n$. On the other hand, by the convex hull property, the supporting hyperplane $\{z\in \R^n:\ell(z)=0\}$ must pass through at least one cylindrical point. That is, there is some $\omega\in P_n$ such that
\ba\label{eq:msupport2}
0=\sum_{i=1}^n\ell^iw^m_{\omega^{-1}(i)}=\sum_{i=m+1}^n\ell^{\omega(i)}\,.
\ea
Combining \eqref{eq:msupport1} and \eqref{eq:msupport2}, we deduce that
\bann
\ell^{\omega(i)}=0\quad\text{for}\quad i=m+1,\dots,n\quad\text{and}\quad \ell^{\omega(i)}\geq 0 \quad\text{for}\quad i=1,\dots,m\,.
\eann
It follows that $\ell$ supports the positive cone, which yields the claim.
\end{proof}

It follows that the distance to the boundary of the positive cone $\Gamma_+:=\{W\in \mathcal{S}(n):W>0\}$ reaches an interior minimum at such a point. The well-know splitting theorem for $\Gamma_+$ (see, for example, \cite[Proposition 4.2.7]{Mantegazza}) now implies that the solution splits off a line. But this is impossible by compactness of $M^n$.
\end{proof}

Next, we prove a geometric Poincar\'e inequality for functions with support compactly contained in the set of non-cylindrical points. To formulate the estimate, we denote by
\bann
\mathrm{Cyl}:=\cup_{m=0}^{n-1}\mathrm{Cyl}_m
\eann
the set of \emph{cylindrical points}, where $\mathrm{Cyl}_m$ is the set of $m$-cylindrical points defined above.%; that is, the set of points $W\in\mathcal{S}(n)$ with a null eigenvalue of multiplicity $m$ and a positive eigenvalue of multiplicity $n-m$.
%\[
%\mathrm{Cyl}_m:=\{r^{-1}\mathrm{diag}(\underbrace{0,\dots,0}_{m\text{\normalfont-times}},\underbrace{1,\dots,1}_{(n-m)\text{\normalfont-times}})\in \mathcal{S}(n):r>0\}
%\]
%is the ray of `standard'  $m$-cylindrical points.

\begin{prop}[Poincar\'e inequality (Cf. {\cite[Lemma 5.4]{Hu84}} and {\cite[Proposition 3.3]{BrHu}})]\label{prop:Poincare}
Let $\Gamma\subset \mathcal{S}(n)$ be an open, $O(n)$-invariant cone satisfying $\overline\Gamma\setminus\{0\}\subset \{W\in \mathcal{S}(n):\tr(W)>0\}$ and $\overline\Gamma \cap\mathrm{Cyl}=\emptyset$. Then there is a constant $\gamma=\gamma(n,\Gamma)>0$ with the following property: Let $X:M^n\to\R^{n+1}$ be a smooth hypersurface and $u\in W^{2,1}(M^n)$ a function for which the set $\{A_{x}:u(x)>0, x\in M^n\}$ is precompact and lies in $\Gamma$. Then, for every $r>0$,
\bann
\gamma\int u^2|A|^2\,d\mu\leq r^{-1}\int|\cd u|^2\,d\mu+(1+r)\int u^2\frac{|\cd A|^2}{H^2}\,d\mu\,.
\eann
\end{prop}
\begin{proof}
Define the tensor
\bann
C:=A\otimes A^2-A^2\otimes A\,.
\eann
Observe that
\bann
|C|^2=2\sum_{i>j}\kappa_i^2\kappa_j^2(\kappa_i-\kappa_j)^2.
\eann
It follows that $C$ vanishes only at cylindrical points. In particular, by homogeneity, and compactness of $\{W\in \overline\Gamma:|W|=1\}$, there is a constant $\gamma=\gamma(n,\Gamma)>0$ such that
\bann
|C|^2\geq\gamma |A|^2H^4
\eann
for all points in the support of $u$. On the other hand, Simons' identity states that
\bann
\cd_{(i}\cd_{j)}A_{kl}-\cd_{(k}\cd_{l)}A_{ij}=C_{ijkl}\,,
\eann
where the brackets indicate symmetrization. Thus,
\bann
\gamma\int u^2 |A|^2\,d\mu\leq{}& \int u^2H^{-4}|C|^2\\
={}&\int u^2H^{-4}C^{ijkl}(\cd_{i}\cd_{j}A_{kl}-\cd_{k}\cd_{l}A_{ij})\\
={}&\int u^2\lb 2H^{-4}C^{ijkl}u^{-1}\cd_iu-4H^{-5}C^{ijkl}\cd_{i}H\right.\\
{}&\left.+H^{-4}\cd_iC^{ijkl}\rb\cd_jA_{kl}\\
{}&-\int u^2\lb 2H^{-4}C^{ijkl}u^{-1}\cd_ku-4H^{-5}C^{ijkl}\cd_{k}H\right.\\
{}&\left.+H^{-4}\cd_kC^{ijkl}\rb\cd_lA_{ij}\,.
\eann
Estimating (using homogeneity and compactness) $|C|\leq c(n,\Gamma)H^3$ and $|\cd C|\leq c(n,\Gamma)H^2|\cd A|$ yields
\bann
\gamma\int u^2 |A|^2\,d\mu\leq{}&c\int u^2\lb2\frac{|\cd u|}{u}\frac{|\cd A|}{H}+\frac{|\cd A|^2}{H^2}\rb
\eann
for a constant $c$ depending only on $n$ and $\Gamma$. The claim now follows from Young's inequality.
\end{proof}

The following identity will play an analogous role for the inscribed curvature $\overline k$.

\begin{prop}[Cf. {\cite[\S 4]{ALM13}}, {\cite[\S 2]{NC2}} and {\cite[\S 3]{Br15}}]\label{prop:kSimons}
Let $X:M^n=\pd\Omega^{n+1}\hookrightarrow \R^{n+1}$ be a smooth, properly embedded, strictly mean convex hypersurface. Then, on the set $\overline M:= \{x\in M:\overline k(x)>\kappa_n(x)\}$,
\ba\label{eq:kSimons}
\frac{1}{2}H%\leq{}&\tr\Big\{ \cd^2\overline k\lb W\,\cdot\,,W\,\cdot\,\rb+\cd_{W\cd\overline k}\,A\lb W\,\cdot\,,W\,\cdot\,\rb\nonumber\\
%{}&-2\lsb \cd\overline k\otimes W\cd\overline k\rsb\lb W\,\cdot\,,W\,\cdot\,\rb+\frac{1}{2}\vert W\cd\overline k\vert^2W\Big\}
\leq{}& \Div\lb W^2\cd\overline k\rb-\inner{W}{\cd_{W^2\cd\overline k}A}+\frac{1}{2}\vert W\cd\overline k\vert^2\tr(W)\,.
\ea
in both the viscosity and the distributional sense, where the tangent bundle endomorphism $W:\overline M\to T^\ast M\otimes TM$ is defined by $W^{-1}:=\overline k\,\mathrm{I}-A$.
\end{prop}
\begin{proof}
We first show that the inequality holds in the viscosity sense. So fix $x_0\in \overline M$ and let $\varphi \in C^\infty(B_{r}(x_0))$, $r>0$, be an upper support for $\overline k$ at $x_0$; that is,
\[
\varphi\geq \overline k\quad \text{on}\quad  B_{r}(x_0) \quad\text{and}\quad \varphi(x_0)=\overline k(x_0)\,.
\]
Now, since $\overline k(x_0)>\kappa_n(x_0)$, there is a point $y_0\in M\setminus\{x_0\}$ such that $\overline k(x_0)=k(x_0,y_0)$. Choosing $r$ possibly smaller (namely, so that $B_r(x_0)\cap B_r(y_0)=\emptyset$), we can arrange that $k\in C^\infty(B_{r}(x_0)\times B_{r}(y_0))$. Thus, the function $\varphi(x,y):=\varphi(x)$ is an upper support for $k$ at $(x_0,y_0)$. Since both functions are smooth, this implies
\bann
-\cd^2\varphi\leq -\cd^2k
\eann
at $(x_0,t_0)$, where the Hessians are over $M^n\times M^n$. To estimate the right hand side, choose local orthonormal coordinates $\{x^i\}$ near $x_0$ and $\{y^i\}$ near $y_0$ and, for any $n\times n$ matrix $\Lambda$, consider
\bann
\cd_{\pd_{x^i}+{\Lambda_i}^p\pd_{y^p}}k={} \frac{2}{d^2}\left(\inner{\pd^x_i-{\Lambda_i}^p\pd^y_p}{\nu_x-kdw}+\inner{dw}{\W^x(\pd_{x^i})}\right)\,,
\eann
where we have defined
\bann
d(x,y,t):=\norm{\X(x,t)-\X(y,t)}\,;&\quad w(x,y,t):= \frac{\X(x,t)-\X(y,t)}{d}\,,
\eann
and
\bann
\partial^x_i:= \frac{\partial \X}{\partial x^i}\,;&\quad \partial^y_i:= \frac{\partial \X}{\partial y^i}\,,
\eann
with sub- and super-scripts $x$ and $y$ denoting quantities relating to the first and second factors respectively.

We now compute the second derivatives.
\ba\label{eq:D2k}
\cd_{\pd_{x^j}+{\Lambda_j}^p\pd_{y^p}}&\cd_{\pd_{x^i}+{\Lambda_i}^p\pd_{y^p}}k\nonumber\\
={}& \frac{2}{d^2}\Big\{\inner{-\W^x{}_{ij}\nu_x +{\Lambda_i}^p{\Lambda_j}^q\W^y{}_{pq}\nu_y}{\nu_x-kdw}\nonumber\\
{}& +\inner{\pd^x_i-{\Lambda_i}^p\pd^y_p}{{\W^x_j}^q\pd^x_q}-\cd_{\pd_{x^j}+{\Lambda_j}^q\pd_{y^q}}k\inner{\pd^x_i-{\Lambda_i}^p\pd^y_p}{dw}\nonumber\\
{}&-k\inner{\pd^x_i-{\Lambda_i}^p\pd^y_p}{\pd^x_j-{\Lambda_j}^q\pd^y_q}\nonumber\\
{}&+\inner{\pd^x_j-{\Lambda_j}^q\pd^y_q}{\W^x_i{}^p\pd^x_p}+\inner{dw}{\cd \W^x_{ij}-(\W^x)^2_{ij}\nu_x}\nonumber\\
&{}-\cd_{\pd_{x^i}+{\Lambda_i}^p\pd_{y^p}}k\inner{\pd^x_j -{\Lambda_j}^q\pd^y_q}{dw}\Big\}\,.
\ea

Next, we use the vanishing of the $y$-derivatives at $y_0$ to determine the tangent plane to $M^n$ at $y_0$:
\begin{lem}[See {\cite[Lemma 6]{ALM13}}]\label{lem:ytangent}
Fix $x_0\in M^n$ and suppose that $\overline k(x_0)=k(x_0,y_0)$ for some $y_0\in M^n\setminus\{x_0\}$. Then
$$
\nu_y=\nu_x-kdw
$$
at $(x_0,y_0)$.
\end{lem}
\begin{proof}[Proof of Lemma \ref{lem:ytangent}]
Since there is an inscribed ball $B\subset\Omega$ of radius $1/k$ touching $\pd\Omega$ at $x_0$ and $y_0$, we obtain
\bann
\nu(y_0)={}& k(x_0,y_0)\lb y_0-\lb x_0- \frac{1}{k(x_0,y_0)}\nu(x_0)\rb\rb \\
={}&(\nu_x-dkw)|_{(x_0,y_0)}\,.
\eann
\end{proof}

Thus, the tangent plane at $y_0$ is the reflection of the tangent plane at $x_0$ about the plane orthogonal to $y_0-x_0$. In particular, we can choose the basis at $y_0$ to be the reflection of the basis at $x_0$:
\ba\label{eq:reflectingtangentplanes}
\pd^y_i=\pd_i^x-2\inner{\pd_i^x}{w}w\quad\text{at}\quad (x_0,y_0)\,.
\ea

Recall now that
\ba\label{eq:Dk}
\pd_i^xk={}&-\frac{2}{d^2}\inner{dw}{(k\mathrm{I}-A^x)\pd^x_i}\,.
\ea
If we choose the basis at $x_0$ so that $A^x$ is diagonal, then
\bann
\frac{2}{d^2}\inner{dw}{\pd_{x^i}}={}&-\frac{\pd_i^xk}{k-\kappa_i^x}\,.
\eann
In particular, this implies (at $(x_0,y_0)$)
\ba\label{eq:dw}
dw={}&\inner{dw}{\nu_x}\nu_x+\sum_{i=1}^n\inner{dw}{\pd^x_i}\pd^x_i\nonumber\\
={}&\frac{d^2}{2}\lb k\nu_x-\sum_{i=1}^n\frac{\pd_{x^i}k}{k-\kappa^x_i}\pd^x_i\rb
\ea
so that
\ba\label{eq:d}
\frac{2}{d^2}=\frac{1}{2}\lb k^2+\sum_{i=1}^n\frac{(\pd_{x^i}k)^2}{(k-\kappa^x_i)^2}\rb\,.
\ea

Applying Lemma \ref{lem:ytangent} and equations \eqref{eq:reflectingtangentplanes}, \eqref{eq:dw} and \eqref{eq:d} to \eqref{eq:D2k} yields, after some calculation,
\ba\label{eq:Dk2}
-\cd_i\cd_j\varphi\leq{}&-\cd_{\pd_{x^j}+{\Lambda_j}^p\pd_{y^p}}\cd_{\pd_{x^i}+{\Lambda_i}^p\pd_{y^p}}k\nonumber\\
\leq{}& k(A^2_x)_{ij}-k^2A^x_{ij}-\cd_{\pd_{x^i}}k\frac{\cd_{\pd_{x^j}}k}{k-\kappa^x_j}-\cd_{\pd_{x^j}}k\frac{\cd_{\pd_{x^i}}k}{k-\kappa^x_i}\nonumber\\
{}&+\frac{1}{2}\lb k^2+\sum_{i=1}^n\frac{(\pd_{x^i}k)^2}{(k-\kappa^x_i)^2}\rb\lb X_{ij}-2\Lambda_i{}^pX_{pj}+\Lambda_i{}^p\Lambda_j{}^qY_{pq}\rb\nonumber\\
{}&+\sum_{p=1}^n\frac{\cd_{\pd_{x^p}}k}{k-\kappa^x_p}\cd_pA^x_{ij}
\ea
at $(x_0,y_0)$, where we have set $X:=k\mathrm{I}-A^x$ and $Y:=k\mathrm{I}-A^y$. Writing
\bann
kA^2_x-k^2A^x=-kA_xX=-k(k\mathrm{I}-X)X,
\eann
the zero order terms can be collected, up to a factor of two, into the matrix
\bann
Z:={}&k^2\lb X-2\Lambda X+\Lambda Y\Lambda^T-2\lb \mathrm{I}-\frac{X}{k}\rb X\rb\,.
\eann
If $Y$ is positive definite at $(x_0,y_0)$, the expression is optimized with the choice $\Lambda:=XY^{-1}$, in which case,
\bann
Z:={}&k^2\lb X-XY^{-1}X-2\lb \mathrm{I}-\frac{X}{k}\rb X\rb\\
={}&kX\lb \mathrm{I}-kY^{-1}+\mathrm{I}-kX^{-1}\rb X\\
={}&kX\lb (Y-k\mathrm{I})Y^{-1}+(X-k\mathrm{I})X^{-1}\rb X\\
={}&-kX\lb A^yY^{-1}+A^xX^{-1}\rb X\,.
\eann
We claim that
\bann
A^yY^{-1}\geq k^{-1}A^y\,.
\eann
Since each of the matrices in the expression can be mutually diagonalized, it suffices to show that
\bann
\frac{\kappa^y_i/k}{1-\kappa^y_i/k}\geq \kappa^y_i/k
\eann
for each $i$, which follows immediately from the fact that $\kappa^y_i/k<1$. Similarly, we obtain
\bann
A^xX^{-1}\geq k^{-1}A^x\,.
\eann
It follows that
\bann
\tr\big(WZW^T\big)\leq -k(H_x+H_y)\leq -kH_x\,,
\eann
where $W:=X^{-1}$. By a straightforward approximation argument, this must also hold when $Y$ is only non-negative definite at $(x_0,y_0)$. Returning to \eqref{eq:Dk2}, we conclude that
\ba\label{eq:evolvekprelim}
\frac{1}{2}H\leq{}&\tr\Big\{ \cd^2\overline k\lb W\,\cdot\,,W\,\cdot\,\rb+\cd_{W\cd\overline k}\,A\lb W\,\cdot\,,W\,\cdot\,\rb\nonumber\\
{}&%-2\lsb \cd\overline k\otimes W\cd\overline k\rsb\lb W\,\cdot\,,W\,\cdot\,\rb
-2W^2(\cd\overline k)\otimes W(\cd\overline k)+\frac{1}{2}\vert W\cd\overline k\vert^2W\Big\}
\ea
in the viscosity sense. To obtain \eqref{eq:kSimons}, observe that
\bann
\cd_kW=%-W\cdot(\cd_k\overline k-\cd_kA)\cdot W=
W\cdot \cd_kA\cdot W-\cd_k\overline k W^2
\eann
so that
\bann
\Div(W^2\cd\overline k)={}&\tr\lb\cd^2\overline k(W\cdot,W\cdot)\rb+\inner{W}{\cd_{W^2\cd\overline k}A}\\
{}&+\inner{W^2}{\cd_{W\cd\overline k}A}-2\inner{W^3(\cd\overline k)}{\cd\overline k}\,.
\eann
The final two terms cancel with the second and third terms of \eqref{eq:evolvekprelim}, which yields the claim.

%\bann
%-W_i{}^pW_j{}^q\cd_p\cd_q\varphi\leq{}& -\frac{1}{2}k(A^x+A^y)_{ij}-W_i{}^pW_j{}^q\lb W_q{}^r\cd_p\varphi\cd_r\varphi+W_p{}^r\cd_p\varphi\cd_r\varphi\rb\\
%{}&+W_i{}^pW_j{}^q\inner{\cd A^x_{pq}}{W(\cd \varphi)}+\norm{W(\cd \varphi)}^2W_{ij}
%\eann
%Taking the trace and discarding $H_y>0$ yields the claim.

To see that the inequality is satisfied in the distributional sense it now suffices, as in the proof of Proposition \ref{prop:evolveG}, to find, for each $x_0\in \overline M$, a smooth lower support $\phi\in C^\infty(B_{r}(x_0))$, $r>0$, for $\overline k$ at $x_0$. Since $\overline k(x_0)>\kappa_n(x_0)$, we have $\overline k(x_0)=k(x_0,y_0)$ for some $y_0\in M$, so that we may take $\phi(x):=k(x,y_0)\leq \overline k(x)$ on a small ball about $x_0$ (i.e. one not containing $y_0$).
\end{proof}

Finally, we recall the following differential inequality for $\overline k$ under mean curvature flow.
\begin{lem}\label{lem:evolvek}
Let $X:M^n\times[0,T)\to\R^{n+1}$ a smooth mean curvature flow of properly embedded hypersurfaces $M^n_t=\pd\Omega_t^{n+1}$. Then
\bann
(\pd_t-\Delta)\overline k\leq |A|^2\overline k-2\inner{\cd\overline k}{W(\cd\overline k)}
\eann
on the set $\overline U:= \{(x,t)\in M^n\times(0,T):\overline k(x,t)>\kappa_n(x,t)\}$ in both the viscosity and the distributional sense, where the tangent bundle endomorphism $W\in \Gamma(\overline U,T^\ast M\otimes TM)$ is defined by $W^{-1}:=\overline k\,\mathrm{I}-A$.
\end{lem}
\begin{proof}
See \cite{ALM13,NC2} and \cite{Br15}.
\end{proof}

\section{Proof of Theorem \ref{thm:pinching}}\label{sec:pinching}

In this section, we prove Theorem \ref{thm:pinching}. By time-translation and scaling invariance of \eqref{eq:MCF}, it suffices to consider the case $t_0=0$ and $R=1$. Let $\Lambda$ be the convex hull of the cylindrical points in $\Gamma$ and let $G_1:M^n\times[0,T)\to\R$ be the distance of $A$ to $\Lambda$. Then
\bann
G_1(x,t):=\max\{-d_{\Lambda}(A_{(x,t)}),0\}\,.
\eann
Thus, by Lemma \ref{eq:evolveG}, $G_1$ satisfies
\bann
(\pd_t-\Delta)G_1\leq |A|^2G_1
\eann
in the distributional sense. Using \eqref{eq:evolvenormA} and Lemma \ref{lem:gradterm}, we can modify $G_1$ to obtain a useful gradient term. Indeed, by Proposition \ref{prop:evolveG}, there is a constant $L=L(n,\Gamma_0)<\infty$ such that $|A|\leq LH$ under the flow. Setting $G_2:=2LH-|A|$, we find
\bann
(\pd_t-\Delta)G_2=|A|^2G_2+\frac{1}{2|A|^3}|A\otimes \cd A-\cd A\otimes A|^2\,.
\eann
If we set
\ba\label{eq:uniformGmconvexity}
G:=%K(G_1,G_2):=
G_1^2/G_2
%\frac{G_1^2}{G_2}\,,
\ea
then a straightforward computation yields%Since $K$ is a convex function of $G_1$ and $G_2$, we find
\ba\label{eq:evolveGgradient}
(\pd_t-\Delta)G%\leq{}&(\pd_t-\Delta)\\
\leq{}& |A|^2G-\frac{G}{2G_2|A|^3}|A\otimes \cd A-\cd A\otimes A|^2
\ea
on the set $\{(x,t)\in M^n\times [0,T):G(x,t)>0\}$ in the distributional sense. Estimating
\bann
G_2|A|\leq 2L^2H^2
\eann
we obtain, from Lemma \ref{lem:gradterm},
\ba\label{eq:evolveGeps}
(\pd_t-\Delta)G\leq |A|^2G-\gamma_1G\frac{|\cd A|^2}{H^2}
\ea
on the set $U_\varepsilon:=\{(x,t)\in M^n\times[0,T):G(x,t)\geq \varepsilon H(x,t)\}$ for some $\gamma_1=\gamma_1(\Gamma_0,\varepsilon)$. 

Now consider, for any $\sigma\in(0,1)$, the function 
\bann
\Ges:=(G-\varepsilon H)H^{\sigma-1}\,.%\frac{G-\varepsilon H}{H}H^\sigma
\eann
We will prove the theorem by bounding $\Ges$ from above for some $\sigma\in(0,1)$.

First observe that, by \ref{eq:evolveGeps},
\ba\label{eq:evolveGes}
(\pd_t-\Delta)\Ges\leq{}&\sigma|A|^2\Ges-\gamma_1H^{\sigma-1}G\frac{|\cd A|^2}{H^2}-\sigma(1-\sigma)\Ges\frac{|\cd H|^2}{H^2}\nonumber\\
{}&+2(1-\sigma)\inner{\cd\Ges}{\frac{\cd H}{H}}\nonumber\\
\leq{}&\sigma|A|^2\Ges-\gamma_1\Ges\frac{|\cd A|^2}{H^2}+2|\cd\Ges|\frac{|\cd H|}{H}
\ea
on $U_\varepsilon$ in the distributional sense. Setting $\Gesp:=\max\{\Ges,0\}$ and applying \eqref{eq:evolvemu}, we obtain (for almost every $t\in[0,T)$)
\bann
\frac{d}{dt}\int \Gesp^pd\mu={}&p\int \Gesp^{p-1}\pd_t\Ges\,d\mu-\int\Gesp^pH^2
\,d\mu\,.
\eann
Discarding the second term on the right, applying \eqref{eq:evolveGes} and integrating the diffusion term by parts, we obtain 
\bann
\frac{d}{dt}\int \Gesp^pd\mu\leq{}&-p(p-1)\int \Gesp^{p-2}|\cd\Ges|^2d\mu-\gamma_1p\int\Gesp^p\frac{|\cd A|^2}{H^2}d\mu\\
{}&+2p\int\Gesp^p\frac{|\cd\Ges|}{\Ges}\frac{|\cd H|}{H}\,d\mu+\sigma p\int\Gesp^p|A|^2\,d\mu\,.
\eann
Young's inequality now yields
\ba\label{eq:k0estmconvexity}
\frac{d}{dt}\int \Gesp^pd\mu\leq {}&-\lb p^2-p^{\frac{3}{2}}-p\rb\int \Gesp^{p-2}|\cd\Ges|^2\,d\mu\nonumber\\
{}&-\lb\gamma_1p-p^{\frac{1}{2}}\rb\int\Gesp^p\frac{|\cd A|^2}{H^2}\,d\mu\nonumber\\
{}&+\sigma p\int\Gesp^p|A|^2\,d\mu\,.
\ea

To estimate the final term, we apply the Poincar\'e inequality, Proposition \ref{prop:Poincare}, to the function $u^2:=\Gesp^p$ with $r=p^{\frac{1}{2}}$. Noting that $|\cd u|^2=\frac{p^2}{4}\Gesp^{p-2}|\cd\Ges|^2$, this yields
\ba\label{eq:Poincareancient}
\gamma_2\int \Gesp^p|A|^2\,d\mu\leq{}& \frac{p^\frac{3}{2}}{4}\int \Gesp^{p-2}|\cd \Ges|^2\,d\mu\nonumber\\
{}&+\lb p^{\frac{1}{2}}+1\rb\int \Gesp^p\frac{|\cd A|^2}{H^2}\,d\mu\,,
\ea
where $\gamma_2>0$ depends only on $n$, $\Gamma_0$ and $\varepsilon$. Applying this to \eqref{eq:k0estmconvexity} yields
\bann
\frac{d}{dt}\int \Gesp^pd\mu\leq {}&-\lb p^2-p^{\frac{3}{2}}-p-\gamma_2^{-1}\sigma p^{\frac{5}{2}}\rb\int \Gesp^{p-2}|\cd\Ges|^2\,d\mu\nonumber\\
{}&-\lb\gamma_1p-p^{\frac{1}{2}}-\gamma_2^{-1}\sigma p(p^{\frac{1}{2}}+1)\rb\int\Gesp^p\frac{|\cd A|^2}{H^2}\,d\mu\,.
\eann
For large $p$ and $\sigma\sim p^{-\frac{1}{2}}$ the right hand side is non-positive.
\begin{prop}\label{prop:mconvexityLpest}
There exists $\ell=\ell(n,\Gamma_0,\varepsilon)>0$ such that, for almost every $t\in [0,T)$,
\bann
\frac{d}{dt}\int \Gesp^pd\mu\leq 0
\eann
for all $p>\ell^{-1}$ and $\sigma\leq \ell p^{-\frac{1}{2}}$.
\end{prop}

%\subsection{The iteration argument}

Just as in \cite{Hu84}, we can now use the Michael--Simon Sobolev inequality and Stampacchia's lemma to extract an $L^\infty$-bound for $\Gesp$ from the $L^p$-estimate.
This yields, for any $\varepsilon>0$, constants $\sigma_0=\sigma_0(n,\Gamma_0,\varepsilon)\in(0,1)$ and $C=C(n,\Gamma_0,\Theta,\varepsilon)<\infty$ such that
\ba
G \leq{}&\varepsilon H+CH^{1-\sigma}\label{eq:sigma0}
\ea
for any $\sigma\in(0,\sigma_0]$. It follows that

\ba
G\leq{}&2\varepsilon H+C_\varepsilon\,,\label{eq:asymptoticGest}
\ea
where $C_\varepsilon=C_\varepsilon(n,\Gamma_0,\Theta,\varepsilon)$. 

To complete the proof, we fix any $\eta>0$ and consider two cases. First, observe that
\bann
G_1\leq \eta H
\eann
whenever $G_1\leq \frac{\eta}{2L}G_2$. On the other hand, if we set $\varepsilon:=\eta^2/4L$, then \eqref{eq:asymptoticGest} yields, wherever $G_1>\frac{\eta}{2L} G_2$,
\bann
G_1\leq{}& \frac{G_2}{G_1}(2\varepsilon H+C_\varepsilon)\\
\leq {}&\frac{2L}{\eta}(2\varepsilon H+ C_\varepsilon)\\
\leq{}& \eta H+C_\eta\,.
\eann
This completes the proof of Theorem \ref{thm:pinching}.

\section{Proof of Theorem \ref{thm:inscribed}}\label{sec:inscribed}

We now prove Theorem \ref{thm:inscribed}. %The case $m=n-1$ was proved in \cite{Br15}, so fix 
Once again, it suffices to consider the case $t_0=0$ and $R=1$. So fix $m\in\{0,\dots,n-2\}$ and set $G_1:=\overline k-\frac{1}{n-m}H$. Recalling Lemma \ref{lem:evolvek}, we obtain
\bann
(\pd-\Delta)G_1\leq |A|^2G_1+Q_1(\cd \overline k)
\eann
on the set $\overline U:=\{(x,t)\in M^n\times(0,T):\overline k(x,t)>\kappa_n(x,t)\}$ in the distributional sense, where, in a principal frame,
\bann
Q_1(\cd \overline k)={}&-2\sum_{i=1}^n\frac{(\cd_i\overline k)^2}{\overline k-\kappa_i}\,.
\eann
This gradient term will be used to control terms involving $\cd\overline k$. 
\begin{lem}\label{lem:kQestmconvex}
There is, for any $\varepsilon>0$, a constant $\gamma=\gamma(n,\varepsilon)>0$ such that
\bann
-2\sum_{i=1}^n\frac{(\cd_i\overline k)^2}{\overline k-\kappa_i}\leq -\gamma\frac{|\cd \overline k|^2}{G}
\eann
on the set $\{(x,t)\in M\times(0,T):G(x,t)\geq\varepsilon H(x,t)\}$.
\end{lem}
\begin{proof} Since, for each $i$, $\overline k>\kappa_i$, we can estimate
\bann
\overline k-\kappa_i\leq n\overline k-H= nG+\frac{m}{n-m}H\leq \lb n+\frac{m}{n-m}\varepsilon^{-1}\rb G\,.
\eann
\end{proof}

To control terms involving $\cd A$, we can make use of Lemma \ref{lem:gradterm} by modifying $G_1$ as in the proof of Theorem \ref{thm:pinching}. So set $G_2:=2LH-|A|$, where $L=L(n,\alpha,\varLambda)$ is chosen so that
\bann
|A|\leq LH \quad \text{and} \quad G_1/G_2\leq 1\,.
\eann
Then the function
\ba\label{eq:inscribedG}
G:=%K(G_1,G_2):=
G_1^2/G_2
\ea
is well-defined and satisfies
\bann
(\pd_t-\Delta)G\leq{}& |A|^2G-4\frac{G}{G_2}\sum_{i=1}^n\frac{(\cd_i\overline k)^2}{\overline k-\kappa_i}\\
{}&-\frac{G}{2G_2|A|^3}|A\otimes \cd A-\cd A\otimes A|^2
\eann
on the set $\overline U$ in the distributional sense. By Lemmas \ref{lem:gradterm} and \ref{lem:kQestmconvex}, there is, for any $\varepsilon>0$, a constant $\gamma_1=\gamma_1(n,\alpha,\varLambda,\varepsilon)$ such that\footnote{Note that the cylindrical point $\mathrm{diag}(0,\dots,0,1)$ is ruled out by the initial pinching condition $\kappa_1+\dots+\kappa_{m+1}\geq \alpha H$, $m\in \{0,1,\dots,n-2\}$.}
\bann
(\pd_t-\Delta)G\leq |A|^2G-\gamma_1 G\lb\frac{|\cd\overline k|^2}{G^2}+\frac{|\cd A|^2}{H^2}\rb
\eann
on the set $\overline U_\varepsilon:=\overline U\cap \{(x,t)\in M^n\times(0,T):G(x,t)>\varepsilon H(x,t)\}$ in the distributional sense.

Next, fix $\varepsilon>0$ and, by \eqref{eq:sigma0}, constants $\sigma_0=\sigma_0(n,\alpha,\varepsilon)\in(0,\frac{1}{2})$ and $K=C(n,\alpha,\Theta,\varepsilon)<\infty$ such that 
\ba\label{eq:msmall}
\kappa_n(x,t)-\frac{1}{n-m}H(x,t)\leq \frac{\varepsilon}{2}H(x,t)+KH^{1-\sigma}
\ea
for all $(x,t)\in M\times[0,T)$ and $\sigma\in (0,\sigma_0)$, and define, for any $\sigma\in (0,\sigma_0)$,
$$
\Gesk:=(G-\varepsilon H)H^{\sigma-1}-K\,.
$$
For notational convenience, we also set
$$
\Ges:=(G-\varepsilon H)H^{\sigma-1}\,.
$$
Then, using \eqref{eq:msmall}, we find
\bann
\overline k-\kappa_n\geq{}&G_1-\frac{\varepsilon}{2}H-KH^{1-\sigma}\\
\geq{}&G-\frac{\varepsilon}{2}H-KH^{1-\sigma}\\
={}&H^{1-\sigma}\Gesk+\frac{\varepsilon}{2}H
\eann
so that, wherever $\Gesk\geq 0$,
\bann
\overline k-\kappa_n\geq{}& \frac{\varepsilon}{2}H\,.
\eann

\begin{comment}
To control terms involving $\cd\overline H$, we need the following lemma.
\begin{lem}\label{lem:kQestmconvex}
Suppose that $\Ges>0$. Then
\ben
\item $\displaystyle (1-\sigma)\frac{|\cd H|}{H}\leq \frac{|\cd\Ges|}{\Ges}+\frac{|\cd \overline k|}{G}$ and
\item $\displaystyle Q(\cd\overline k)\leq -n(1+\varepsilon^{-1})\frac{|\cd \overline k|^2}{G}$.
\een
\end{lem}
\begin{proof}
To obtain the first estimate, differentiate $\Ges$ and rearrange terms to obtain
\bann
\frac{\cd H}{H}={}&\frac{H^{\sigma-1}\cd\overline k-\cd\Ges}{(1-\sigma)\Ges+\lb\frac{1}{n-m}+\varepsilon\rb H^\sigma}\\
={}&\frac{\cd\overline k}{(1-\sigma)G+\lb\frac{1}{n-m}+\sigma \varepsilon\rb H}-\frac{\cd\Ges}{(1-\sigma)\Ges+\lb\frac{1}{n-m}+\varepsilon\rb H^\sigma}\,.
\eann
The claim follows.

For the second claim, consider,
\bann
\overline k-\kappa_i\leq n\overline k-H= nG+\frac{m}{n-m}H\leq nG+\frac{m}{n-m}\varepsilon^{-1} G\,.
\eann
The claim follows.
\end{proof}
\end{comment}

We will show that $\Gesk$ is bounded from above for some $\sigma\in(0,\sigma_0)$. Observe first that $\Gesk$ satisfies the differential inequality
\ba\label{eq:evolvekGesmconvex}
(\pd_t-\Delta)\Gesk\leq{}& \sigma|A|^2\Ges-\gamma_1GH^{\sigma-1}\lb\frac{|\cd\overline k|^2}{G^2}+\frac{|\cd A|^2}{H^2}\rb\nonumber\\
{}&+2(1-\sigma)\inner{\cd\Ges}{\frac{\cd H}{H}}
\ea
on the support of $\Gesk$ in the distributional sense. 

Setting $\Geskp:=\max\{\Gesk,0\}$, we will prove the following analogue of Proposition \ref{prop:mconvexityLpest}.
\begin{prop}[$L^p$-estimate]\label{prop:kintestmconvex}
There exists a constant $\ell>0$, which depends only on  $n$, $\alpha$, $\varLambda$, and $\varepsilon$, such that
\bann
\frac{d}{dt}\int \Geskp^p\,d\mu\leq \sigma K^p\int |A|^2d\mu
\eann
for all $p\geq \ell^{-1}$ and $\sigma\leq \ell p^{-\frac{1}{2}}$.
\end{prop}
\begin{proof}
Applying \eqref{eq:evolvekGesmconvex} and integrating the diffusion term by parts, we obtain
\bann
\frac{d}{dt}\int \Geskp^pd\mu\leq{}&-p(p-1)\int\Geskp^{p-2}|\cd\Ges|^2\,d\mu\\
{}&+\sigma p\int\Geskp^{p-1}\Ges|A|^2\\
{}&-\gamma_1p\int\Geskp^{p-1}GH^{\sigma-1}\lb\frac{|\cd\overline k|^2}{G^2}+\frac{|\cd A|^2}{H^2}\rb\,d\mu\\
{}&+2p(1-\sigma)\int\Geskp^{p}\frac{|\cd\Ges|}{\Gesk}\frac{|\cd H|}{H}\,d\mu\,.
\eann

Estimating $\Gesk\leq GH^{1-\sigma}$ and, by Young's inequality,
\bann
pK\Gesk^{p-1}\leq K^p+(p-1)\Gesk^{p}
\eann
and
\bann
2\frac{|\cd\Ges|}{\Gesk}\frac{|\cd H|}{H}\leq p^{\frac{1}{2}}\frac{|\cd\Ges|^2}{\Gesk^2}+p^{-\frac{1}{2}}\frac{|\cd H|^2}{H^2}\,,
\eann
we obtain
\ba
\frac{d}{dt}\int \Geskp^pd\mu\leq {}&-p\lb p-p^{\frac{1}{2}}-1\rb\int\Geskp^{p-2}|\cd\Ges|^2d\mu\nonumber\\
{}&-\gamma_1 p\int \Geskp^{p}\frac{|\cd \overline k|^2}{G^2}d\mu\nonumber\\
{}&-\lb \gamma_1 p-p^{\frac{1}{2}}\rb\int \Geskp^{p}\frac{|\cd A|^2}{H^2}d\mu\nonumber\\
{}&+2\sigma p\int\Geskp^{p}|A|^2d\mu+\sigma K^p \int |A|^2d\mu\,.\label{eq:Poincareremainskmconvex}
\ea

To estimate the penultimate term, we make use of Proposition \ref{prop:kSimons}.
\begin{prop}[Poincar\'e inequality]\label{prop:kPoincaremconvex}
There is a constant $\gamma_2>0$, which depends only on $n$, $\varLambda$ and $\varepsilon$, such that
\bann
\gamma_2\int \Geskp^p|A|^2\,d\mu\leq{}& p^{\frac{3}{2}}\int \Geskp^{p-2}|\cd\Ges|^2d\mu\\
{}&+(p^{\frac{1}{2}}+1)\int \Geskp^{p}\frac{|\cd\overline k|^2}{G^2}d\mu\\
{}&+\int \Geskp^{p}\frac{|\cd A|^2}{H^2}d\mu\,.
\eann
\end{prop}
\begin{proof}
Fix $\gamma=\gamma(n,\varLambda)>0$ so that $2\gamma|A|^2\leq H^2$. Then, by \eqref{eq:kSimons},
\bann
\gamma \int \Geskp^p{}&|A|^2\,d\mu\leq \frac{1}{2}\int \Geskp^pH^2\,d\mu\\
\leq{}& \int\Geskp^pH\Big(\Div\lb W^2\cd\overline k\rb-\inner{W}{\cd_{W^2\cd\overline k}A}\\
{}&\qquad\qquad+\frac{1}{2}\vert W\cd\overline k\vert^2\tr(W)\Big) d\mu\,.
\eann
Integration by parts now yields
\bann
\gamma \int &\Geskp^p|A|^2\,d\mu\\
{}&\quad\leq -\int\Geskp^pH^2\lb p\inner{\frac{\cd \Ges}{\Gesk}}{\frac{W^2\cd\overline k}{H}}+\inner{\frac{\cd H}{H}}{\frac{W^2\cd\overline k}{H}}\right.\\
{}&\qquad\left.+\inner{W}{\frac{\cd_{W^2\cd\overline k}A}{H}}-\frac{1}{2}|W\cd\overline k|^2\frac{\tr(W)}{H}\rb d\mu\,.
\eann
Estimating $|W|\leq C(n,\varepsilon)H^{-1}$ and $G\leq C(n,\varLambda)H$, we obtain
\bann
\gamma \int \Geskp^p|A|^2\,d\mu\leq{}& C\int\Geskp^p\lb p\frac{|\cd\Ges|}{\Gesk}\frac{|\cd\overline k|}{G}\right.\\
{}&\qquad\qquad\qquad+\left.\frac{|\cd\overline k|}{G}\frac{|\cd A|}{H}+\frac{|\cd\overline k|^2}{G^2}\rb d\mu\,,
\eann
where $C=C(n,\varLambda,\varepsilon)$. The claim now follows from Young's inequality.
\end{proof}

Applying Proposition \ref{prop:kPoincaremconvex} to the inequality \eqref{eq:Poincareremainskmconvex} yields

\bann
\frac{d}{dt}\int \Geskp^pd\mu\leq {}&-p\lb p-2\gamma_2^{-1}\sigma p^{\frac{3}{2}}-p^{\frac{1}{2}}-1\rb\int\Geskp^{p-2}|\cd\Ges|^2d\mu\\
{}&-\lb\gamma_1 p-2\gamma_2^{-1}\sigma p\lb p^{\frac{1}{2}}+1\rb\rb\int \Geskp^{p}\frac{|\cd \overline k|^2}{G^2}d\mu\\
{}&-\lb \gamma_1 p-p^{\frac{1}{2}}-2\gamma_2^{-1}\sigma p\rb\int \Geskp^{p}\frac{|\cd A|^2}{H^2}d\mu\\
{}&+\sigma K^p \int |A|^2d\mu\,.
\eann
This completes the proof of Proposition \ref{prop:kintestmconvex}.
\end{proof}

The iteration argument leading to an upper bound for $\Gesk$ now proceeds similarly as in \cite{Hu84}. The proof is then completed by estimating as in the final argument of the proof of Theorem \ref{thm:pinching}.

\section{Ancient solutions}\label{sec:ancient}

%\input{ancient_new_section.tex}

%Using an observation of Huisken and Sinestrari \cite{HuSi15}, we can also obtain optimal pinching estimates for ancient solutions of the mean curvature flow. 
First, we prove Theorem \ref{thm:pinchingancient}.

\begin{proof}[Proof of Theorem \ref{thm:pinchingancient}]
Choose $L=L(n,\Gamma_0)<\infty$ such that 
\[
|A|\leq LH\quad\text{and}\quad -d_\Lambda(A)\leq L^{\frac{1}{2}}H\,,
\]
where $d_\Lambda$ denotes the signed distance to the boundary of $\Lambda$ (see \eqref{eq:distance}), and, as in the proof of Theorem \ref{thm:pinching}, define a function $G:M^n\times(-\infty,1)\to \R$ via
\bann
G:=\frac{\max\{-d_\Lambda(A),0\}^2}{2LH-|A|}\,.
\eann
Set also, for any $\varepsilon>0$ and $\sigma\in(0,1)$,
\[
\Ges:=(G-\varepsilon H)H^{\sigma-1}\,.
\]
Then, proceeding as in the proof of Theorem \ref{thm:pinching}, we obtain the inequalities
\ba\label{eq:pinchingancient1}
\frac{d}{dt}\int \Gesp^pd\mu\leq {}&-\lb p^2-p^{\frac{3}{2}}-p\rb\int \Gesp^{p-2}|\Ges|^2\,d\mu\nonumber\\
{}&-\lb\gamma_1p-p^{\frac{1}{2}}\rb\int\Gesp^p\frac{|\cd A|^2}{H^2}\,d\mu\nonumber\\
{}&+\sigma p\int\Gesp^p|A|^2\,d\mu
\ea
(cf. \eqref{eq:k0estmconvexity}) and 
\ba\label{eq:pinchingancient2}
\gamma_2\int \Gesp^p|A|^2\,d\mu\leq{}& \frac{p^\frac{3}{2}}{4}\int \Gesp^{p-2}|\cd \Ges|^2\,d\mu\nonumber\\
{}&+\lb 1+p^{\frac{1}{2}}\rb\int \Gesp^p\frac{|\cd A|^2}{H^2}\,d\mu
\ea
(cf. \eqref{eq:Poincareancient}) for suitable constants $\gamma_1$ and $\gamma_2$ which depend only on $n$, $\Gamma_0$ and $\varepsilon$, where $\Gesp:=\max\{\Ges,0\}$. We used these to conclude that
\bann
\frac{d}{dt}\int \Gesp^p\,d\mu\leq 0
\eann
for $p$ sufficiently large and $\sigma$ sufficiently small, of the order $p^{-\frac{1}{2}}$; however, choosing $\sigma$ a little smaller (but still of order $p^{-\frac{1}{2}}$), we obtain  the following.
%\bann
%\frac{d}{dt}\int (\Ges)_+^p\,d\mu\leq -\sigma p\int (\Ges)_+^p|A|^2\,d\mu\,.
%\eann
\begin{lem}
There exists $\ell>0$, which depends only on $n$, $\Gamma_0$ and $\varepsilon$, such that
\ba\label{eq:intestconvexityancient}
\frac{d}{dt}\int \Gesp^p\,d\mu\leq -\sigma p\int \Gesp^p|A|^2\,d\mu
\ea
for all $p>\ell^{-1}$ and $\sigma<\ell p^{-\frac{1}{2}}$.
\end{lem}
%\begin{proof}
%Choosing $p$ and $\sigma$ similarly as in the hypotheses, we can obtain from \eqref{eq:convexityancient22} the inequality
%\bann
%\frac{d}{dt}\int (\Ges)_+^p\,d\mu\leq {}&-\frac{1}{2}p^2\int(\Ges)_+^{p-2}|\cd\Ges|^2\,d\mu\\
%{}&-\frac{1}{2}p\int(\Ges)_+^p\frac{|\cd A|^2}{H^2}\,d\mu\,.
%\eann
%Using \eqref{eq:convexityancient21}, we can now choose $\sigma$ slightly smaller, but still of the order $p^{-\frac{1}{2}}$ to get the result.
%\end{proof}
We can now proceed exactly as in \cite[Theorem 5.2]{HuSi15}: Since $\sigma$ is of the order $p^{-\frac{1}{2}}$, we can still arrange that $\sigma p>2n+1$ if $p$ is sufficiently large. In that case, $\delta:=\frac{2}{\sigma p+1}<\frac{1}{n+1}<1$. Thus, noting also that $\Ges\leq H^\sigma$, we may estimate
\bann
\int \Gesp^p\,d\mu={}&\int \Gesp^{p(1-\delta)}\Gesp^{\delta p}\,d\mu\\
\leq{}&\int \Gesp^{p(1-\delta)}H^{\delta \sigma p}\,d\mu\\
%={}&\int (\Gesp^{p}H^2)^{1-\delta}H^{\delta \sigma p-2(1-\delta)}\,d\mu\\
={}&\int (\Gesp^{p}H^2)^{1-\delta}H^{\delta}\,d\mu\\
\leq{}&\lb \int \Gesp^{p}H^2\,d\mu\rb^{1-\delta}\lb\int H\,d\mu\rb^\delta
\eann
for $p$ sufficiently large. Applying \ref{eq:intestconvexityancient}, and recalling the algebraic inequality $n|A|^2\geq H^2$, this yields
\bann
\frac{d}{dt}\int \Gesp^p\,d\mu%\leq{}&-\sigma p\int \Gesp^p|A|^2\,d\mu\\
\leq{}&-(2n+1)\int \Gesp^p|A|^2\,d\mu\\
\leq{}&-\int \Gesp^pH^2\,d\mu\\
\leq{}&-\lb\int \Gesp^p\,d\mu\rb^{\frac{1}{1-\delta}}\lb \int H\,d\mu\rb^{-\frac{\delta}{1-\delta}}\\
\leq{}&-\lb\int \Gesp^p\,d\mu\rb^{1+\frac{2}{\sigma p-1}}\lb \int H\,d\mu\rb^{-\frac{2}{\sigma p-1}}\,.%\\
%={}&-c^{-2}\varphi^{1+\beta}\psi^{-\beta}\,,
\eann
Setting
\bann
\varphi(t):=\int \Gesp(x,t)^p\,d\mu_t(x)\,,\quad \psi(t):=\int H(x,t)\,d\mu_t(x)
\eann
and $\beta:=\frac{2}{\sigma p-1}$, this becomes
\ba\label{eq:phiinequalityconvexityancient2}
\frac{d}{dt}\lb\varphi^{-\beta}\rb=-\beta \varphi^{-\beta-1}\frac{d}{dt}\varphi\geq \beta\psi^{-\beta}\,.
\ea
Note that, if $\varphi(s)=0$ for some $s\in (-\infty,1)$, then $\varphi(t)=0$ for all $t\in [s,1)$ since the inequality $G\leq\varepsilon H$ is preserved under the flow. Set
\bann
\tau_0:=\sup\{s\in (-\infty,0]:\varphi(t)>0 \text{ for all } t<s\}\,.
\eann
Suppose that $\tau_0>-\infty$. Then \eqref{eq:phiinequalityconvexityancient2} and H\"older's inequality yield
\ba\label{phiinequalityconvexityancient3}
\varphi^{-\beta}(\tau)-\varphi^{-\beta}(s)\geq{}& \beta\int_{s}^{\tau}\psi^{-\beta}(t)dt\nonumber\\
\geq{}& \beta(\tau-s)^{1+\beta}\lb\int_{s}^{\tau}\psi(t) dt\rb^{-\beta}
\ea
for any $s<\tau<\tau_0$. Applying the bounded rescaled volume hypothesis \eqref{eq:brv}, we conclude
\bann
\varphi^{\beta}(\tau)\leq C\frac{(1-s)^{\beta(n+1)}}{(\tau-s)^{1+\beta}}
\eann
for some $C<\infty$ and all $s<\tau<\tau_0$. But $0<n\beta<1$, so that, taking $s\to-\infty$, we obtain $\varphi(\tau)=0$, a contradiction. We are forced to conclude that $\tau_0=-\infty$. It follows that $G\geq -\varepsilon H$ for any $\varepsilon>0$, and hence $A_{(x,t)}\in \Lambda$ for all $(x,t)\in M^n\times(-\infty,1)$. The rigidity statement now follows from Proposition \ref{prop:splitting}.%Finally, if $A_{(x_0,t_0)}\in \pd\Lambda$ for some interior point $(x_0,t_0)\in M^n\times(-\infty,1)$ then %, by Lemma \ref{lem:evolveG} and the strong maximum principle, $A_{(x,t)}\in \pd\Lambda$ for all $(x,t)\in M^n\times(-\infty,t_0]$. 
%Proposition \ref{prop:splitting} implies that $M^n_t$ is a shrinking sphere.
\end{proof}

The following Lemma provides sufficient conditions for the bounded rescaled volume hypothesis.
\begin{lem}[Cf. \cite{HuSi15}]\label{lem:weakHbound}
Let $X:M^n\times(-\infty,1)\to\R^{n+1}$, $n\geq 2$, be a compact, connected, mean convex, embedded ancient solution of \eqref{eq:MCF}. Suppose that one of the following hold:
\ben
\item $\displaystyle \limsup_{t\to-\infty}\int H^n\,d\mu<\infty$,
\item $\displaystyle \limsup_{t\to-\infty}\max_{M^n\times\{t\}}H<\infty$,
\item $\displaystyle \liminf_{t\to-\infty}\min_{M^n\times\{t\}}\frac{\kappa_1}{H}>0$ or
\item $\displaystyle |\cd H|^2\leq \frac{1}{n-1}|\AA|^2H^2$.
\een
Then $M_t$ has bounded rescaled volume.
\end{lem}
\begin{proof}
The first three claims are proved just as in \cite{HuSi15}: For the first case, set
\[
\varLambda:=\sup_{t\in(-\infty,0]}\lb\int H^n\,d\mu\rb^{\frac{1}{n}}<\infty\,.
\]
Using H\"older's inequality and \eqref{eq:evolvemu}, observe that
\bann
\frac{d}{dt}\mu(M^n)^{\frac{2}{n}}={}&\frac{2}{n}\mu(M^n)^{\frac{2-n}{n}}\frac{d}{dt}\mu(M^n)\\
={}&-\frac{2}{n}\mu(M^n)^{\frac{2-n}{n}}\int H^2\,d\mu\\
\geq{}&-\frac{2}{n}\lb\int H^{n}\,d\mu\rb^{\frac{2}{n}}\\
\geq{}&-\frac{2}{n}\varLambda^{2}
\eann
for $t<0$. Integrating yields
\bann
\mu_t(M^n)\leq C(1-t)^{\frac{n}{2}}
\eann
for $t<0$. Applying the isoperimetric inequality for compact subsets of $\R^{n+1}$ then yields
\bann
(n+1)\omega_{n+1}^\frac{1}{n+1}\,|\Omega_t|\leq \mu_t(M^n)^{\frac{n+1}{n}}\leq C(1-t)^{\frac{n+1}{2}}\,,
\eann
where $\omega_{n+1}$ is the volume of the unit ball in $\R^{n+1}$. The claim follows.

For the second case, the speed bound yields a bound, for $t<0$,
\bann
\rho_+(t)\leq C(1-t)
\eann
for the circumradius, which immediately yields the desired bound, for $t<0$,
\bann
|\Omega_t|\leq C(1-t)^{n+1}
\eann
for the enclosed volume.

For the third case, we apply the area formula to obtain
\bann
\int H^n\,d\mu\leq \varepsilon^{-n}\int K\,d\mu=\varepsilon^{-n}\mathrm{Area}(S^n)
\eann
for some $\varepsilon>0$, where $K$ is the Gauss curvature, which reduces to case (1).

The final claim also follows from the first, after integrating the identity
\bann
\frac{d}{dt}\int H^n\,d\mu=n\int H^n\lb H^2|\AA|^2-(n-1)\frac{|\cd H|^2}{H^2}\rb\,d\mu\geq 0\,.
\eann
\end{proof}

We now prove Corollary \ref{cor:shrinkingsphere}.

\begin{proof}[Proof of Corollary \ref{cor:shrinkingsphere}]
The equivalence of (1) and (2) was proved by Huisken and Sinestrari \cite{HuSi15}, as were the remaining cases under the additional hypothesis that $M_t$ is convex for all $t$. Thus, by Theorem \ref{thm:pinchingancient}, it suffices to prove that the solutions in the remaining cases have bounded rescaled volume. This is immediate in case (3). Case (4) follows from the comparison principle, which, comparing the solution with shrinking sphere solutions, yields, for all $t<0$,
\bann
\rho_-(t)\leq C\sqrt{1-t}
\eann
for some $C<\infty$, reducing the hypothesis to that of case (3). To deal with case (5), we recall that a well-known comparison argument for \eqref{eq:evolveH} yields
\bann
\min_{M^n\times\{t\}}H\leq C\sqrt{\frac{n}{2(1-t)}}
\eann
for some $C<\infty$. It then follows from the hypothesis that $H$ is bounded for $t<0$ and the claim follows from Lemma \ref{lem:weakHbound}. The proof is similar for case (6). In the final case, we estimate, using H\"older's inequality,
\bann
|\Omega_t|-|\Omega_0|={}&\int_t^0\hspace{-3mm}\int H(\cdot,s)\,d\mu_s\,ds\\
\leq{}&\lb\int_t^0\hspace{-3mm}\int H^2(\cdot,s)\,d\mu_s\,ds\rb^{\frac{1}{2}}\lb\int_t^0\hspace{-3mm}\int d\mu_s\,ds\rb^{\frac{1}{2}}\\
={}&\big(\mu_t(M^n)-\mu_0(M^n)\big)^{\frac{1}{2}}\lb\int_t^0\mu_s(M^n)\,ds\rb^{\frac{1}{2}}\\
\leq{}&\big(\mu_t(M^n)-\mu_0(M^n)\big)^{\frac{1}{2}}\sqrt{-t}\,\mu_t(M^n)^{\frac{1}{2}}\,.
\eann
The reverse isoperimetric inequality now yields
\bann
|\Omega_t|%\leq{}& C\sqrt{1-t}\,\mu_t(M^n)\\
\leq{}& C\sqrt{1-t}\,|\Omega_t|^{\frac{n}{n+1}}
\eann
for $t<0$. Rearranging yields 
\bann
|\Omega_t|\leq C(1-t)^{\frac{n+1}{2}}\,,
\eann
which implies the claim.
\end{proof}

\begin{rem}
In some cases, we can already obtain information about ancient solutions from Theorem \ref{thm:pinching} (or, indeed, from the previously known estimates \cite{Hu84}, \cite{HuSi99b}). For example, in \cite{HuSi15} it was proved (see equation (2.10) and Lemmas 4.1 and 4.4 and the proof of Theorem 4.2) that compact, convex ancient solutions satisfying a reverse-isoperimetric inequality automatically satisfy
\bann
\mu_{t}(M^n)^{\frac{1}{n}}\leq C\sqrt{T-t}\,,
\eann
\bann
\max_{M^n\times\{t\}}H\leq \frac{C}{\sqrt{1-t}}
\eann
and
\bann
A\geq \varepsilon Hg\,.
\eann
Theorem \ref{thm:pinching} then implies that the solution is a shrinking sphere.

Another nice situation is the case of surfaces evolving in  $\R^3$. In that case, we have
\bann
-\frac{d}{dt}\mu_t(M^2)=\int H^2\,d\mu\geq \int_{M^2_+} H^2\,d\mu\geq 4\int_{M^2_+}K\,d\mu\geq 4\sigma_2\,,
\eann
where $M_+^2$ is the contact set of $M$ and $\sigma_2:=\mathrm{Area}(S^2)$. Integrating, we find
\bann
4\sigma_2(T-t)\leq \mu_t(M^2)\,.
\eann
It now follows from Theorem \ref{thm:pinching} that mean convex, type-I ancient solutions $X:M^2\times(-\infty,1)\to\R^3$ of \mcf satisfying
\bann
\liminf_{t\to-\infty}\min_{M^2\times\{t\}} \frac{\kappa_1}{H}>-\infty
\eann
are strictly convex. Arguing as in \cite[Theorem 4.2]{HuSi15}  then implies that $M^2_t$ is uniformly convex, and hence, applying Theorem \ref{thm:pinching} once more, a shrinking sphere.
\end{rem}

Using Corollary \ref{cor:mconvexityancient}, Theorem \ref{thm:inscribedancient} can be proved by the same method as Theorem \ref{thm:pinchingancient}.
\begin{proof}[Proof of Theorem \ref{thm:inscribedancient}]
Define the function $G$ as in \ref{eq:inscribedG} and set, for any $\varepsilon>0$ and $\sigma\in(0,1)$,
\bann
\Ges:=(G-\varepsilon H)H^{\sigma-1}\,.
\eann
Since, by Corollary \ref{cor:mconvexityancient}, $\kappa_n-\frac{1}{n-m}H\leq 0$ (i.e. we can take $K=0$ in the definition of $\Gesk$), we can proceed as in the proofs of Theorem \ref{thm:inscribed} and Theorem \ref{thm:pinchingancient} to obtain the inequality
\bann
\frac{d}{dt}\int \Gesp^p\,d\mu\leq-\sigma p\int \Gesp^p|A|^2\,d\mu\,,
\eann
where $\Gesp:=\max\{\Ges,0\}$. We then continue as in the proof of Theorem \ref{thm:pinchingancient} to obtain $G\leq 0$; that is,
\bann
\overline k\leq \frac{1}{n-m}H\,.
\eann
To obtain the rigidity statement, assume that $\overline k=\frac{1}{n-m}H$ at some interior point. It follows from the evolution equations for $H$ \eqref{eq:evolveH} and $\overline k$ \eqref{eq:evolveinscribedk} and the strong maximum principle \cite{DL04} that $\overline k\equiv \frac{1}{n-m}H$. If $\overline k=\kappa_n$ at some point, the claim follows from Proposition \ref{prop:splitting} as in Theorem \ref{thm:pinchingancient}. So suppose that $\overline k>\kappa_n$ everywhere. Then, together, the evolution equation for $H$ \eqref{eq:evolveH} and the presence of the gradient term in the evolution equation for $\overline k$ on the set $\{\overline k>\kappa_n\}$ (Lemma \ref{lem:evolvek}) yield the conclusion $\cd \overline k\equiv 0$, and hence $\cd H\equiv 0$. It follows that the solution is a shrinking sphere \cite{Al39}. %In fact, there can only be one such component: Suppose, to the contrary, that there are at least two components, $\pd B^{n+1}_{\sqrt{2n(t_1-t)}}(p_1)$ and $\pd B^{n+1}_{\sqrt{2n(t_2-t)}}(p_2)$. Fix any time $t<\min\{t_1,t_2\}$. Note that $B^{n+1}_{\sqrt{2n(t_1-t)}}(p_1)$ cannot intersect $B^{n+1}_{\sqrt{2n(t_2-t)}}(p_2)$ non-trivially, since this would violate the non-collapsing condition. It follows that $B^{n+1}_{\sqrt{2n(t_1-t)}}(p_1)$ and $B^{n+1}_{\sqrt{2n(t_2-t)}}(p_2)$ cannot be disjoint either, since in that case the two enclosed regions would intersect non-trivially at some earlier time. The final case $B^{n+1}_{\sqrt{2n(t_1-t)}}(p_1)\subsetneq B^{n+1}_{\sqrt{2n(t_2-t)}}(p_2)$ (or vice versa) is also ruled out since $\frac{\overline k}{H}\to \frac{1}{n}$ on the outer sphere $\pd B^{n+1}_{\sqrt{2n(t_2-t)}}(p_2)$ as $t\to t_1$ but is strictly less than $\frac{1}{n}$ for $t<t_1$, violating $\pd_t \frac{\overline k}{H}\equiv 0$.
\end{proof}

Combining the argument in \cite{Br15} with the convexity estimate of Corollary \ref{cor:convexityancient}, Theorem \ref{thm:exscribedancient} is proved similarly.

%\printbibliography
\bibliographystyle{plain}
\bibliography{bib}

\end{document}